\documentclass[12pt]{article}

\usepackage{graphicx}
\usepackage{amsmath,amsthm,amssymb,color}
\usepackage[noadjust,sort]{cite}
 \usepackage[normalem]{ulem}

\normalsize

\theoremstyle{plain}
\newtheorem{Theorem}{Theorem}[section] %
\newtheorem{Lemma}[Theorem]{Lemma}

\newtheorem{Corollary}[Theorem]{Corollary}

\theoremstyle{definition}
\newtheorem{Remark}[Theorem]{Remark}

\def\calF{\mathcal{F}}
\def\calC{\mathcal{C}}
\def\calL{\mathcal{L}}
\def\calH{\mathcal{H}}
\def\calW{\mathcal{W}}

\theoremstyle{definition}

\newtheorem{Problem}{Problem}[section]

{\par\noindent{\it Proof of}} % êîìàíäû äëÿ \begin
{\hfill$\vspace{5mm}\scriptstyle\blacksquare$} % êîìàíäû äëÿ \end

\numberwithin{equation}{section} %×òîáû íóìåðàöèÿ â êàæäîé ñåêöèè áûëà íåçàâèñèìîé
\numberwithin{figure}{section} %×òîáû íóìåðàöèÿ â êàæäîé ñåêöèè áûëà íåçàâèñèìîé
\numberwithin{table}{section} %×òîáû íóìåðàöèÿ â êàæäîé ñåêöèè áûëà íåçàâèñèìîé

\def\supp{\operatorname{supp}}

\def\Reals{{\mathbb{R}}}
\def\Complexes{{\mathbb{C}}}

\def\Naturals{{\mathbb{N}}}

\def\st{\,:\,}

\def\dfrac#1#2{\lower0.15ex\hbox{\large$\textstyle\frac{#1}{#2}$}}

\begin{document}

\setcounter{page}{1}

\markboth{M.I. Isaev, R.G. Novikov}{H\"older-logarithmic stability in inverse problems for the reconstruction
from the Fourier transform in the ball}

\title{
Stability  estimates for reconstruction
 from the Fourier transform  on the ball\thanks{The first author's research is  supported by    the Australian Research  Council  Discovery Early Career Researcher Award DE200101045. 
}
}
\date{}
\author{ 
Mikhail Isaev\\
\small School of Mathematics\\[-0.8ex]
\small Monash University\\[-0.8ex]
\small Clayton, VIC, Australia\\
%\small Moscow Institute of Physics and Technology\\[-0.9ex]
%\small Dolgoprudny, Moscow Region, Russia\\[-0.3ex]
\small\texttt{mikhail.isaev@monash.edu}
\and
Roman G. Novikov\\
\small CMAP, CNRS, Ecole Polytechnique\\[-0.8ex]
\small Institut Polytechnique de Paris\\[-0.8ex]
\small Palaiseau, France\\
\small IEPT RAS, Moscow, Russia\\
\small\texttt{novikov@cmap.polytechnique.fr}
}

\maketitle
%{\bf Abstract}
%\begin{abstract}
%Many important inverse problems are exponentially ill-posed in general, which constitutes a severe difficulty for numerical treatments. However, a stable reconstruction of the unknown parameter might still be possible in some cases when the parameter is well-behaved.  Countless results in the literature confirm improved stability under various additional a-priori assumptions.  In fact, the behaviour of the stability bounds can change dramatically from the logarithmic type to the H\"{o}lder type or even, under some strong assumptions, to the Lipchitz type. In this work, we illustrate such transitions with an example from the classical Fourier analysis.
%\end{abstract}

\begin{abstract}
	We prove  H\"{o}lder-logarithmic stability estimates
	for the problem of finding an  integrable function $v$ on $\Reals^d$
	with a super-exponential decay at infinity  
	 from its Fourier transform  $\calF v$
	given on the ball $B_r$.  These estimates arise from a H\"{o}lder-stable extrapolation of   $\calF v$ from $B_r$ to a larger ball. 
	We also present instability examples showing an optimality of our  results. 
	
%	
%	 for the problem of finding a sufficiently regular compactly supported function $v$ on $\Reals^d$ from its Fourier transform  $\calF v$
%	given on $[-r,r]^d$. 
%	This estimate  relies on a H\"{o}lder stable continuation of $\calF v$
%		from  $[-r,r]^d$  to a larger domain. The related reconstruction procedures are based on truncated series of Chebyshev polynomials. 
%		 We also give an explicit example showing optimality of our stability estimates.
%		
\end{abstract}

\section{Introduction}

We consider the classical Fourier transform $\mathcal{F}$  defined by 
\begin{equation*}
 	\mathcal{F} v (\xi) := 
 	\dfrac{1}{(2\pi)^d}\int\limits_{\mathbb{R}^d} e^{i\xi x} v(x) dx, \ \ \ \xi\in \mathbb{R}^d,
 \end{equation*}
where $v$ is a test function on $\Reals^d$ and  $d\geq 1$. 
%The analysis of this transform is one of the 
%most developed domains of modern mathematics. 
% In particular, it is well known that if  $v$ 
%is compactly supported then $\calF v $ is analytic on $\Reals^d$. Thus, $\calF v$ and, consequently, $v$ are uniquely determined  
%by  the values of $\calF v $ within any open non-empty domain.  However, in the case of noisy  data  
%such reconstruction is believed to be very hard. 
%In the present paper we investigate how much the stability improves with respect to  the size of the domain  where $\calF v$ is known and with respect to the regularity of $v$.
%The inverse problem  
%For a  sufficiently regular compactly supported 
%function $v$ on $\mathbb{R}^d$, 
Let  
\[
	B_r := \left\{\xi\in \mathbb{R}^d :  |\xi| < r\right\}, \qquad
 	\text{where $r>0$.}
\]
Assume that $v$ is integrable 
and,  for some  $N,\sigma>0$ and $\nu\geq 1$, we have that 
%\begin{equation}\label{eq:ass}
   \begin{equation} \label{eq:ass}
   Q_v(\lambda) :=   \dfrac{1}{(2\pi)^d}  \int\limits_{\Reals^d} e^{\lambda|x|}|v(x)|dx \leq  N\exp\left(\sigma \lambda^{\nu}\right), 
   \qquad 
   \text{for all $\lambda\geq 0$.}
  \end{equation}   
%
% and that, for some $\eta>0$,
%\begin{equation}\label{ass:eq}
%	Q(\lambda) :=   \dfrac{1}{(2\pi)^d}  \int\limits_{\Reals^d} e^{\lambda|x|}|v(x)|dx =\exp\left(   O\left(\lambda^{\eta}\right)\right)
%	\qquad \text{as $\lambda  \rightarrow +\infty$.}
%\end{equation}
\begin{Remark}\label{Rem1}
In particular, for the case of $\nu=1$, the class of functions satisfying \eqref{eq:ass}
 includes all functions  $v$ with $\supp v \subset  B_\sigma$ and $  \dfrac{1}{(2\pi)^d} \|v\|_{\calL^1(\Reals^d)} \leq N$.
%\in  \mathbb{R}_{>0} := \{x\in \Reals \,:\, x> 0\}$
Furthermore,
if $v$ is such that
	\[	
		|v(x)| \leq  C \exp\left(- \mu |x|^{\eta}\right) 
		\text{ for some $\mu, C>0$ and $\eta> 1$,}
\]  
then assumption  \eqref{eq:ass} holds with    $\nu:= \frac{\eta}{\eta-1}$ and with some positive constants $\sigma =  \sigma(\mu,\eta)$ and $N = N(C,\mu,\eta,d)$.
\end{Remark}

Under  assumption \eqref{eq:ass}, we consider the following two problems:
\begin{Problem}\label{Problem1}
	Given $\mathcal{F}v$ on the ball $B_r$. Find $v$.
\end{Problem}
\begin{Problem}\label{Problem2}
Given $\mathcal{F}v$ on the ball $B_r$.  Find  $\mathcal{F} v$ on $B_R$, where $R> r$.
\end{Problem}

Problems \ref{Problem1} and \ref{Problem2}
 are fundamental in the theory of inverse coefficient problems.
% 
% arise, in particular,
% in the framework of inverse problems.
 For example, Problem~\ref{Problem1} with $r=2\sqrt{E}$  
can be regarded as  a linearized inverse scattering problem for the Schr\"odinger
equation with potential $v$ at fixed positive energy $E$, for $d\geq 2 $, and on the the energy interval $[0,E]$, for $d\geq 1$.
More details can be found in  \cite[Section 4]{Novikov2020}.
Problem~\ref{Problem1} with $r=\omega_0$
also arises in a multi-frequency inverse source problem for the homogeneous Helmholtz equation with frequencies $\omega\in [0, \omega_0]$;
see Bao et al.~\cite[Section 3]{BLT2010} for more details. 
In addition, in many cases, Problem  \ref{Problem2} is  an essential step  for solving Problem \ref{Problem1}.  
For more applications related to Problems \ref{Problem1} and \ref{Problem2} in the case of compactly supported $v$, see  \cite{IN2020} and references therein.

The present work continues the studies of  our recent article  \cite{IN2020}, which considers   the case of compactly supported functions $v$.
Besides, in \cite{IN2020},   we deal with 
reconstructions of $\calF v$ on $[-R,R]^d$   and  $v$ on $\Reals^d$    from $\calF v$ given on  the cube $[-r,r]^d$,   in place of the balls $B_R$ and  $B_r$.  Due to the equivalence of $\|\cdot\|_2$ -norm and $\|\cdot\|_\infty$-norm in $\Reals^d$, these formulations   are essentially equivalent, but $B_R$ and $B_r$ are more natural   in the context of   inverse problems.

      In the present work,  under assumption \eqref{eq:ass}, we give H\"older-logarithmic stability estimates for Problem \ref{Problem1} in the norm 
     of  $\calL^\infty(\Reals^d)$  and of $\calH^s(\Reals^d)$, for any real $s$; see Section~\ref{S:rec}. 
     (Note that the stability estimates of \cite{IN2020} are given in the norm of $\calL^2(\Reals^d)$ only.)
     In addition, 
   %  under assumption \eqref{eq:ass}, 
      we obtain
      H\"older stability estimates for Problem \ref{Problem2}
      in the norm of $\calL^\infty(B_R)$; see Section \ref{S:con}.
    The related reconstruction procedures are also given; see 
    Sections~\ref{S:reconstruction} and~\ref{S:rec}.
      Besides, we present examples showing an optimality of our stability estimates and   reconstruction procedures; see Section \ref{S:instability}.

 \section{Reconstruction procedures}\label{S:reconstruction}

 Let 
 $\calF^{-1}$ be the classical inverse Fourier transform defined by 
\begin{equation*}
	  \mathcal{F}^{-1} [u](x) := \int \limits_{\mathbb{R}^d}u(\xi) e^{-i\xi x} d \xi, \qquad x\in \Reals^d.
\end{equation*}      
For a given $r>0$, we consider  the following    family of  extrapolations   $\mathcal{C}_{R,n}: \calL^\infty(B_r) \rightarrow \calL^\infty(B_R)$,  depending on two parameters 
$R\geq r$ and 
$n \in \mathbb{N}:= \{0,1,\ldots\}$.
%
%We use  the truncated series  
%of  Chebyshev polynomials $T_k$,  for $k \in \mathbb{N}:=\{n \in \mathbb{Z} \st n\geq  0\}$ to define 
%the family of  continuation operators  $\mathcal{C}_{R,n}$ 
%Consider the following family of continuations $\mathcal{C}_{R,n}$ 
 For a  function $w$  on $B_r$  (for example, such that $w \approx \mathcal{F}v|_{B_r}$), we define
\begin{equation}\label{def_C}
	[{\mathcal C}_{R,n} w](\xi) := 
	\begin{cases}
	 w(\xi), & \xi \in B_r,\\\displaystyle
	 \sum \limits_{k=0}\limits^{n-1} 
	a_k\left(\frac{\xi}{|\xi|}\right) T_k\left(\frac{|\xi|}{r}\right), & \xi \in B_R \setminus B_r,
	\\
		0,	& \xi \in \mathbb{R}^d\setminus B_R,
	\end{cases}
\end{equation}
where $\xi = |\xi|\theta$ and, for   $\theta \in S^{d-1}$,
\begin{equation}\label{def_a}
	a_k(\theta)  = a_k[w](\theta):=
	\begin{cases}
	 \displaystyle \frac{1}{\pi}\int\limits_{-r}\limits^r \frac{w(t\theta)}{\sqrt{r^2- t^2}} dt, & \text{if } k=0,\\
	 \displaystyle \frac{2}{\pi}\int\limits_{-r}\limits^r \frac{w(t\theta)T_k\left(\frac{t}{r}\right)}{\sqrt{r^2- t^2}} dt,
	 &\text{otherwise.}
	\end{cases}
\end{equation}
In the above, $(T_k)_{k\in \mathbb{N}}$ stand  for the Chebyshev polynomials on $\Reals$, which can be defined  
by   $T_k(t):= \cos(k\operatorname{arccos}(t))$ if $t\in [-1,1]$ and extended  to $|t|>1$ in a natural way.
 For $n=0$, the sum in \eqref{def_C} is taken to be $0$. 
Note that formulas \eqref{def_C} and \eqref{def_a} are correctly defined for almost all $\xi$ and $\theta$ under the assumption that $w \in \calL^\infty(B_r)$.

%Recall that the Chebyshev polynomials $T_k$,  for $k \in \mathbb{N}:=\{n \in \mathbb{Z} \st n\geq  0\}$,  can be defined by  
%\begin{equation} \label{def_Cheb}
%	T_k(x):= 
%	\begin{cases}
%	 \cos(k\operatorname{arccos}(x)), & |x|\leq 1,\\
%	\dfrac{(x-\sqrt{x^2-1})^k + (x+\sqrt{x^2-1})^k}{2}, &|x|>1.
%\end{cases}
%\end{equation}

Suppose $w  \approx \calF v |_{B_r}$.
The transforms  $\mathcal{C}_{R,n} w$  on  $B_R$ can be considered as a family of reconstruction procedures for Problem \ref{Problem2}.
The transforms  $ \mathcal{F}^{-1} \mathcal{C}_{R,n} w $  on  $\Reals^d$ can be considered as a family of reconstruction procedures for Problem \ref{Problem1}. 

 In Section \ref{S:rec}, we give stability estimates for 
 Problem \ref{Problem1} arising from the reconstructions  $\mathcal{F}^{-1} \mathcal{C}_{R,n}$; see Theorem \ref{Theorem1} and Theorem \ref{Theorem2}.
  In Section \ref{S:con}, we give stability estimates for 
 Problem \ref{Problem2}  arising from the extrapolations 
 $\mathcal{C}_{R,n}$;  see  Lemma \ref{Lemma_C}, Theorem \ref{Theorem3}, and Corollary \ref{Corollary_C}.

\section{Stability estimates for Problem \ref{Problem1}}\label{S:rec}
%
%All aforementioned  results (Theorem \ref{Theorem2},  Theorem \ref{Theorem3}, and  Corollary~\ref{Corollary_C}) share the following assumptions in common:  
We will assume that 
the  unknown function $v:\Reals^d \rightarrow \Complexes$ 
satisfies  \eqref{eq:ass}  for some  $N,\sigma>0$ and $\nu\geq 1$ and
%  and sufficientl
%   \begin{equation} \label{eq:ass}
%   Q_v(\lambda)\leq  N\exp\left(\sigma \lambda^{\nu}\right) 
%  \end{equation}   
%     and 
 the given data $w$ is such that, for some $\delta, r>0$,
\begin{equation}\label{eq:ass2}
     \|w - \calF v\|_{\calL^\infty(B_r)} \leq \delta  <N, 
   \end{equation}
%\end{equation}
where $\mathcal{F}$ is  the Fourier transform.  Note that
if   \eqref{eq:ass} holds then, for any $\xi \in \Reals^d$,
\begin{equation}\label{eq:N-delta}
	|\calF v(\xi)| \leq \frac{1}{(2\pi)^d} \int_{\Reals^d}|v(x)|dx  = Q_v(0) \leq N.
\end{equation}
This explains the condition $\delta<N$ in assumption \eqref{eq:ass2}. Indeed, if the noise level $\delta$ is greater than $N$ then the given data $w$ tells  about  $v$  as little as the trivial function $w_0\equiv 0$.

To 
achieve optimal stability  bounds,
  the parameters $R$ and $n$
   in the reconstruction $\calF^{-1}\calC_{R,n}$    have to be chosen carefully depending on $N, \delta, r, \sigma$.   For any $\tau \in [0,1]$, let
\begin{equation}\label{def_L}
			L_{\tau}(\delta) = L_\tau(N,\delta,r,\sigma,\nu):= \max\left\{1,\, \frac{1}{2}\left(\frac{(1-\tau) \ln  \frac{N}{\delta}}{\sigma r^{\nu}}\right)^\tau\right\}.
	\end{equation}
 Here and thereafter, we assume  $0<\delta<N$. 
Using \eqref{def_C}, define 
	\begin{equation}\label{def_Cstar}
		\calC^*_{\tau,\delta} := {\mathcal C}_{R_\tau(\delta),n_\tau(\delta)},
	\end{equation}
	 where 
	\begin{equation}\label{def_R}
		\begin{aligned}
		R_\tau(\delta) &= R_\tau(N, \delta,r,\sigma,\nu):= r L_{\tau}(\delta),\\
		n_\tau(\delta) &= n_\tau(N,\delta,r, \sigma,\nu):= 
		\begin{cases}
		\displaystyle \left\lceil\frac{ (2-\tau)\ln\frac{N}{\delta}}{ \ln 2 + \frac{1}{\tau \nu}\ln (2L_{\tau}(\delta))}
		\right\rceil, 
		&\text{ if    $\tau>0$,}\\
		0,&\text{otherwise.}
		\end{cases}
		\end{aligned}
	\end{equation}
	and $\lceil\cdot\rceil$ denotes the ceiling of a real number. 
	Let 
	\begin{equation}\label{def_c}
		c(d) :=   \int\limits_{\partial B_1} 1\, dx = \frac{d \pi^{d/2}}{\Gamma(\frac{d}{2}+1)}.
	\end{equation}
	
%\begin{equation}
%	\begin{aligned}
%	W^{m,p}(\mathbb{R}^3) = \{w:\  \partial^J w \in L^1(\mathbb{R}^3),\  |J| \leq m \},\ m \in \mathbb{N}\cup\{0\},
%\\
%	||w||_{m,1} = \max\limits_{|J|\leq m} ||\partial^J w||_{L^1(\mathbb{R}^3)},
%\\
%J \in (\mathbb{N}\cup \{0\})^3,\ |J| = J_1+J_2+J_3,\ \partial^J w(x) 
%= \frac{\partial^{|J|} w(x)}{\partial x_1^{J_1}\partial x_1^{J_2} \partial x_3^{J_3}},
%\end{aligned}
%\end{equation}

%\subsection{Stability of the reconstruction $\mathcal{F}^{-1}{\mathcal C}^{\tau}$}\label{S:v-recon}

Our first result is  a stability estimate for Problem \ref{Problem1} in the norm $\calL^\infty(\Reals^d)$.
In addition to \eqref{eq:ass}, we assume also   that   $v \in \calW^m(\Reals^d)$,
 where  the space $ \calW^m(\Reals^d)$, $m\geq 0$, and its norm are defined by
\begin{equation*}\label{def_Wm}
	\begin{aligned}
	\calW^{m}(\mathbb{R}^d) &:= \left\{u \in \calL^1(\mathbb{R}^d)\st (1+|\xi|^2)^{\frac{m}{2}} \mathcal{F} u \in  \calL^\infty(\mathbb{R}^d) \right\},
\\
	||u||_{\calW^{m}(\mathbb{R}^d)} &:= \left\| (1+|\xi|^2)^{\frac{m}{2}} \mathcal{F} u \right\|_{\calL^\infty(\mathbb{R}^d)}.
\end{aligned}
\end{equation*}
 We note that for integer $m$ the space $\calW^{m}(\mathbb{R}^d)$ contains  the standard Sobolev space $\calW^{m,1}(\mathbb{R}^d)$  of $m$-times smooth functions in $\calL^1$ on $\mathbb{R}^d$.

%Let
%\begin{equation}\label{def_c12}
%	c_1 = \int \limits_{B_1}1\,  dx = \frac{\pi^{d/2}}{\Gamma(\frac{d}{2}+1) }, \ \ \ \ \ c_2 = \int\limits_{\partial B_1} 1\, dx = \frac{d \pi^{d/2}}{\Gamma(\frac{d}{2}+1) },
%\end{equation}
%where  $\Gamma$ is the gamma function.
\begin{Theorem}\label{Theorem1} 
Let  the assumptions of \eqref{eq:ass} and \eqref{eq:ass2} hold for some $N,\sigma, r, \delta>0$  and $\nu\geq 1$. 
	Assume  also that  $v \in \calW^m(\mathbb{R}^d)$, 
	%\sout{, $\|v\|_{\calW_m(\mathbb{R}^d)} \leq N_{\calW^m}$,} 
	for some real $m>d$, and   that $ \|v\|_{\calW^m(\mathbb{R}^d)} \leq \gamma_1$.
	%\sout{and $N_{\calW^m}>0$}. 
	Then, for any $\alpha$ such that $0\leq \alpha\leq 1$, the following estimate holds:
			\begin{equation}\label{eq_Theorem1}
		\begin{aligned}
		\left\|v - \mathcal{F}^{-1}\calC^*_{\tau,\delta} w \right\|_{\calL^{\infty}(\mathbb{R}^d)}
		 &\leq  
		   \dfrac{8c(d) }{d}  N^{1-\alpha} r^d   \left({L}_{\tau}(\delta)\right)^{d+1} 
		   \delta^{\alpha}\\
		  &+  \dfrac{c(d) }  {m-d}     \gamma_1	 	r^{-m+d} \left({L}_{\tau}(\delta)\right)^{-m+d},
		 \end{aligned}
	\end{equation}
	where
	$\tau  = 1 - \sqrt{1 - (1-\alpha) \nu^{-1}}$ 
	and  ${L}_{\tau}(\delta)$, $\calC^*_{\tau,\delta}$, $c(d)$ 
	are defined by \eqref{def_L}, \eqref{def_Cstar}, \eqref{def_c}.  
	In particular,    for any   $\beta_1$ such that 
	$ 0<\dfrac{\beta_1}{m-d}  <1- \sqrt{1- \nu^{-1}}$, we have
	%$\sigma$, $r$ $m$,  $d$ and $\tau$
	%for any $\tau$ such that $0<\tau < 1- \sqrt{1- \nu^{-1}}$,  the following estimate holds:
	\begin{equation}\label{log-1}
 			\left\|v - \mathcal{F}^{-1}\calC^*_{\tau,\delta} w \right\|_{\calL^{\infty}(\mathbb{R}^d)}
 			\leq    c_1 \left(\ln (3+\delta^{-1})\right)^{-\beta_1},
 	\end{equation}
 	where  $\tau = \dfrac{\beta_1}{m-d}$ and $c_1= c_1(N,\sigma, \nu,     r,   m,  \gamma_1, d, \beta_1) $ is a positive constant. 
%	and 
%	$\alpha := 1- \nu \tau(2-\tau)$.
%%	where 
%%	\begin{equation*} %\label{def_c12}
%%	c_1 := \int \limits_{B_1}1\,  dx = \frac{\pi^{d/2}}{\Gamma(\frac{d}{2}+1) }, \ \ \ \ \ c_2 := \int\limits_{\partial B_1} 1\, dx = \frac{d \pi^{d/2}}{\Gamma(\frac{d}{2}+1) }.
%%\end{equation*}
\end{Theorem}

Our second result is  a stability estimate for Problem \ref{Problem1} in the norm $\calH^s(\Reals^d)$. Recall that
  the Sobolev space $\calH^s(\Reals^d)$, $s\in \Reals$, and its norm can be defined by
\begin{equation*}\label{def_Hm}
	\begin{aligned}
	\calH^{s}(\mathbb{R}^d) &:= \left\{u \in \calL^2(\mathbb{R}^d):\ \mathcal{F}^{-1} (1+|\xi|^2)^{\frac{s}{2}} \mathcal{F} u \in  \calL^2(\mathbb{R}^d) \right\},
\\
	||u||_{\calH^{s}(\mathbb{R}^d)} &:= \left\|\mathcal{F}^{-1} (1+|\xi|^2)^{\frac{s}{2}} \mathcal{F} u \right\|_{\calL^2(\mathbb{R}^d)}.
\end{aligned}
\end{equation*}

\begin{Theorem} \label{Theorem2}
Let  the assumptions of \eqref{eq:ass} and \eqref{eq:ass2} hold for some $N,\sigma, r, \delta>0$  and $\nu\geq 1$. 
Assume also that	 $v \in \calH^m(\mathbb{R}^d)$,  for some real $m \geq -\dfrac d  2$, and
that  $ \|v\|_{\calH^m(\mathbb{R}^d)} \leq \gamma_2$.  Then, for any $\alpha \in [0,1]$	 and any $s < m$,  the following estimate holds:
	 \begin{equation} \label{eq_Theorem2}
		\begin{aligned}
		\|v - \mathcal{F}^{-1}\calC^*_{\tau,\delta}w \|_{\calH^{s}(\mathbb{R}^d)}
		  &\leq   8 (2\pi)^{d/2} c(d)N^{1-\alpha }\left( \int_{0}^{r L_\tau(\delta)} (1+t^2)^s t^{d-1}dt \right)^{1/2}     {L}_{\tau}(\delta)
		  \, \delta^{\alpha} 
		   \\
		  &+ \gamma_2  r^{-m+s} \left({L}_{\tau}(\delta)\right)^{-m+s},
		 \end{aligned}
		 \end{equation}
where
	$\tau := 1 - \sqrt{1 - (1-\alpha) \nu^{-1}}$  and ${L}_{\tau}(\delta)$, $\calC^*_{\tau,\delta}$,  
	%$R_\tau(\delta)$, 
	$c(d)$  
	are defined by \eqref{def_L},  \eqref{def_Cstar}, 
	%\eqref{def_R}, 
	\eqref{def_c}.
		In particular,    for any   $\beta_2$ such that 
	$ 0<\dfrac{\beta_2}{m-s}  <1- \sqrt{1- \nu^{-1}}$, we have
	%$\sigma$, $r$ $m$,  $d$ and $\tau$
	%for any $\tau$ such that $0<\tau < 1- \sqrt{1- \nu^{-1}}$,  the following estimate holds:
	\begin{equation}\label{log-2}
 			\left\|v - \mathcal{F}^{-1}\calC^*_{\tau,\delta} w \right\|_{\calH^{s}(\mathbb{R}^d)}
 			\leq    c_2 \left(\ln (3+\delta^{-1})\right)^{-\beta_2},
 	\end{equation}
 	where  $\tau = \dfrac{\beta_2}{m-s}$ and $c_2= c_2(N, \sigma, \nu, r,   m, s, \gamma_2, d, \beta_2) $ is a positive constant. 
\end{Theorem}

The proofs of Theorems \ref{Theorem1} and \ref{Theorem2} are given in Section \ref{S:proofs}.
The first terms of the right-hand side in  estimates \eqref{eq_Theorem1} and   \eqref{eq_Theorem2} correspond 
  to the error caused by the H\"older stable extrapolation of the noisy data $w$  from $B_r$ to $B_{R_\tau(\delta)}$ and the second (logarithmic) terms correspond to the error  caused by ignoring the values of  $\calF v$ outside  
  $B_{R_\tau(\delta)}$;  see Section \ref{S:proofs} for more details.

%Theorems \ref{Theorem1} and \ref{Theorem2} directly imply the following corollary.
Let $N$, $\sigma$, $\nu$, $r$, $m$, $\gamma_1$, $\gamma_2$, $d$ be fixed.  Then estimates \eqref{log-1} and \eqref{log-2} used for $v := v_1-v_2$ and $w:=w_0\equiv 0$
yield the following corollary.

\begin{Corollary}\label{Corollary}
	Let   $v_1$ and   $v_2$  
	be such that $v:=v_1 - v_2$ 
	satisfies   \eqref{eq:ass} for some $N,\sigma>0$ and $\nu\geq 1$.
	Let $\tau$  be such that $0<\tau < 1- \sqrt{1- \nu^{-1}}$.
	Then the following  bounds hold.
	\begin{itemize}
		\item[(a)] If $v_1  - v_2  \in  \calW^m(\mathbb{R}^d)$, for some real $m<d$,
		and $\|v_1  - v_2\|_{\calW^m(\mathbb{R}^d)} \leq \gamma_1$,  then
\begin{equation}\label{eq:cor1}
 			\left\|v_1 - v_2 \right\|_{\calL^{\infty}(\mathbb{R}^d)}
 			\leq    c_1 \left(\ln \left(3+ 
 			 \frac{1}{\|\calF v_1 - \calF v_2\|_{\calL^{\infty}(B_r)}}
 			\right)\right)^{-\beta_1},
 	\end{equation}		
 	where $\beta_1 = \tau(m-d)$ and $c_1$ is the constant of  \eqref{log-1}.
 			
 			\item[(b)] If $v_1  - v_2 \in  \calH^m(\mathbb{R}^d)$  for some real $m \geq -\dfrac d  2 $, 
 		$\|v_1  - v_2\|_{\calH^m(\mathbb{R}^d)} \leq \gamma_2$, and
 			  $s < m$,  then
\begin{equation} \label{eq:cor2}
 			\left\|v_1 - v_2 \right\|_{ \calH^s(\mathbb{R}^d)}
 			\leq    c_2 \left(\ln \left(3+ 
 			 \frac{1}{\|\calF v_1 - \calF v_2\|_{\calL^{\infty}(B_r)}}
 			\right)\right)^{-\beta_2},
 	\end{equation}
 	where $\beta_2 = \tau(m-s)$ and $c_2$ is the constant of  \eqref{log-2}.
\end{itemize}
\end{Corollary}

One can see that the estimates     of Theorem \ref{Theorem1}, Theorem \ref{Theorem2}, and Corollary \ref{Corollary} are available 
for any $\beta_1$, $\beta_2$ such that:
\begin{equation}\label{b-max} 
\begin{aligned}
	0<\beta_1<\beta_1^{\max} :=   \left( 1- \sqrt{1- \nu^{-1}}\right) (m-d),
	\\
	0<\beta_2< \beta_2^{\max} :=   \left( 1- \sqrt{1- \nu^{-1}}\right) (m-s).
\end{aligned}
\end{equation}
In particular,  for $\nu=1$, we have that  $\beta_1^{\max}= m-d$ and $\beta_2^{\max} = m-s$.  
 In Section \ref{S:instability},  
      we present instability examples showing  that
      \begin{itemize}
     \item   if $\nu=1$  then 
     \eqref{eq:cor1} is impossible for
        $\beta_1>m$ and \eqref{eq:cor2} is impossible for $\beta_2>m$ (when $d\geq 2$ and $s=0$)
      and  for $\beta_2>m+\frac12 $ (when $d=1$ and $s=0$);
      \item if $\nu=2$ then  
     \eqref{eq:cor1} and \eqref{eq:cor2} (with $s=0$) are impossible for
     $\beta_1,\beta_2>m/2$.
                  \end{itemize}
 These examples show that the logarithmic bounds 
 \eqref{eq:cor1} and \eqref{eq:cor2} 
 are rather optimal with respect to the values 
  of the exponents $\beta_1$ and $\beta_2$. Consequently, in this respect, it is impossible to essentially   improve the  stability bounds of  Theorems \ref{Theorem1} and \ref{Theorem2}, even using any other reconstruction procedure (for example, based on a  more advanced basis instead of Chebyshev polynomials).

 Moreover, observe that, for the case $\nu=1$, 
 the values of $\beta_1^{\max}$ and $\beta_2^{\max}$ are very close to  the best possible: for example, we determined that  $\beta_2 =m$ is indeed  the threshold value  for \eqref{eq:cor2}  when $d\geq 2$ and $s=0$.  
However, we do not know whether 
the claimed exponents  $\beta_1^{\max}, \beta_2^{\max} = 
\left( 1- \sqrt{1- \nu^{-1}}\right) m +O(1)$
 are also that close to optimal for $\nu>1$. Our instability examples   for $\nu=2$ imply that  they can not exceed $m/2$, but there   is still a gap from $m/2$ down to $\left(1-\dfrac1{\sqrt{2}}\right)m$.

 Theorem \ref{Theorem1}, Theorem \ref{Theorem2}, and Corollary \ref{Corollary} illustrate   similar stability behaviour in more complicated non-linear  inverse problems. 
In fact,  the relationship is  closer than a mere illustration taking into account that the monochromatic reconstruction from the scattering amplitude  in the Born approximation is reduced to Problem \ref{Problem1}.  
In particular, estimates  \eqref{log-2}, \eqref{eq:cor1} and \eqref{eq:cor2} (with $\nu=1$)  should be compared  with the results   
on the monochromatic inverse scattering problem  obtained by H\"ahner, Hohage \cite{HH2001}, Isaev,  Novikov  \cite[Theorem 1.2]{IN2013++} and Hohage, Weidling~\cite{HW2017} under the assumption that 
  $v$ is a compactly supported sufficiently regular function on $\Reals^3$.
 More precisely, for  this case, 
estimate   \eqref{eq:cor1} with 
  $m>3$ and $\beta_1 = \dfrac{m-3}{3}$  is similar to   \cite[Theorem 1.2]{IN2013++};
 estimate  \eqref{eq:cor2} with $s=0$, $m>\frac32$, and $\beta_2 = \dfrac{m}{m+3}$ is similar to \cite[Theorem 1.2]{HH2001};
  estimates   \eqref{log-2},  \eqref{eq:cor2} with $m>7/2$ and some appropriate $\beta_2 \in (0,1)$ are similar to  \cite[Corollary 1.4]{HW2017}.     
For other known results on logarithmic and  H\"older-logarithmic stability  in inverse problems,  see  also
Alessandrini~\cite{Alessandrini1988}, Bao et al.~\cite{BLT2010},
 Isaev~\cite{Isaev2013+},  Isakov~\cite{Isakov2011}, 
Novikov~\cite{Novikov2011}, Santacesaria~\cite{Santacesaria2015}  and references  therein.

As observed above,      
logarithmic  and H\"older-logarithmic stability   
 was  established  
 for many different inverse problems. 
 However, to our knowledge, even for the compactly supported case,   the estimates of 
   Theorem~\ref{Theorem1}, Theorem~\ref{Theorem2} (with  $\alpha<1$) and Corollary~\ref{Corollary}   are implied by none of  results given in the literature 
   %(including the aforementionned results of \cite{HH2001,IN2013++, HW2017}, which are the most related)
   before the recent work  \cite{IN2020}.  The related results of \cite{IN2020} are essentially equivalent  to  the special case of \eqref{eq_Theorem2},  \eqref{log-2}, \eqref{eq:cor2} when    $v \in \calH^m(\Reals^d)$ is  compactly supported,  
$m$ is a positive integer, and $s=0$.

\section{Examples of exponential instability for  Problem~\ref{Problem1}}\label{S:instability}
  
  First, we recall the results from \cite[Section 6]{IN2020}. Let  $A$ and $B$ be open bounded domains in $\Reals^d$, $d\geq 1$. Then, for any fixed positive integer $m$  and positive $\gamma$, we give 
   examples of real-valued functions $v_n \in  C^m(\Reals^d)$  such that 
%   \begin{itemize}
%   	\item $	\supp(v_n)  \subseteq  A$,
%   	\item   $\|v_n\|_{C^m(\Reals^d)} \leq \gamma$,
%   \end{itemize}
    \begin{equation}\label{ass_compact}
	\supp(v_n)  \subseteq  A, \qquad  \qquad      \|v_n\|_{C^m(\Reals^d)} \leq \gamma,
%\|\calF v_n \|_{\calL^\infty(B)}< 1,
%	 \nonumber  \\
%	 \|v\|_{\calL^{2}(\Reals^d)} &> c\left( \ln  \frac{1}{\|\calF v \|_{\calL^\infty(B)}}\right)^{-\mu}. \label{T:est}
\end{equation}
and the following asymptotics hold as $n \rightarrow +\infty$:
\begin{align*}
	 \left\|\calF  v_n  \right\|_{\calL^\infty(B )} &= O(e^{-n}), 
	 \qquad 	\|v_n\|_{\calL^{\infty}(\Reals^d)} =  \Omega(n^{-m}),
  \\
 	 \|v_n\|_{\calL^{2}(\Reals^d)}&=  
 	 \begin{cases}
 	 	\Omega(n^{-m}), & \text{for $d\geq2$,}
 	 	\\
 	 	 \Omega(n^{-m-\frac12}), & \text{for $d=1$.}
 	 \end{cases}
\end{align*}
Recall that 
for two sequences of real numbers $a_n$ and $b_n$, we say %$a_n=\omega(b_n)$ or 
$a_n=\Omega(b_n)$  if $a_n>0$ always and $b_n = O(a_n)$.
It follows from the above that,  
for any $\beta_1>m$ and   any constant $c_1>0$,  
\begin{equation}\label{instable:inf}
	\|v_n \|_{\calL^\infty(\mathbb{R}^d)} > c_1 \left( \ln \left(3+ \left\|\mathcal{F} v_n  \right\|_{\calL^\infty(B_r)}^{-1} \right)\right)^{-\beta_1}, 
\end{equation}
  when $n$ is sufficiently large.
 Similarly, for any $\beta_2$, where $\beta_2>m$ if $d\geq2$  and 
 $\beta_2>m+\frac12$ if  $d=1$,   and for any   constant $c_2>0$,  we have that 
 \begin{equation}\label{instable:L2}
 	\|v_n \|_{\calL^2(\mathbb{R}^d)} > 
 	c_2 \left( \ln \left(3+ \left\|\mathcal{F} v_n \right\|_{\calL^\infty(B_r)}^{-1} \right)
 	\right)^{-\beta_2},
\end{equation}
  when $n$ is sufficiently large.

Condition \eqref{ass_compact} and instability estimates   \eqref{instable:inf} and \eqref{instable:L2}
show an optimality (or nearly optimality) of the exponent $\beta_1$
in  stability estimates  \eqref{eq_Theorem1}, \eqref{log-1}, \eqref{eq:cor1}
 and of the exponent $\beta_2$ in  stability estimates 
  \eqref{eq_Theorem2}, \eqref{log-2}, \eqref{eq:cor2}  with $s=0$.   
  %  for $\nu=1$  (which includes  the compactly supported  case).
%  Recall  the definitions of  $\beta_1^{\max}$ and $\beta_2^{\max}$ given in    \eqref{b-max}.
Recall that    Theorem~\ref{Theorem1}, Theorem~\ref{Theorem2}, and Corollary~\ref{Corollary}  require $\beta_1< \beta_1^{\max}$, $\beta_2<\beta_2^{\max}$, where 
  $\beta_1^{\max}$ and $\beta_2^{\max}$ are defined in    \eqref{b-max}. In particular, for $\nu=1$ (which includes  the compactly supported  case), we have that 
    $\beta_1^{\max} = m-d$ and  $\beta_2^{\max} =m$ (for $s=0$), which are  close to the infima of the exponents $\beta_1$ and $\beta_2$    in \eqref{instable:inf} and \eqref{instable:L2}.

    However, $\beta_1^{\max}(\nu)$ and 
 $\beta_2^{\max}(\nu)$ decrease to $0$ as $\nu\rightarrow +\infty$.     In particular,  for $\nu$ noticeable greater than $1$,
the instability behaviour exhibiting by  the functions $v_n$ recalled above
 become much less tight with respect to $ \beta_1^{\max}(\nu)$ and $\beta_2^{\max}(\nu)$. 
 This motivates us  to construct  other explicit examples of exponential instability for 
 Problem~\ref{Problem1}, which are  non-compactly supported  and  provide   considerably smaller exponents $\beta_1$ and $\beta_2$   in the instability estimates in comparison with   $v_n$.

	For $k \in \Reals^d$,  integer  $m>0$, and real $\varepsilon >0$, consider the functions $	v_{k, m, \varepsilon}$ defined by
\begin{equation}\label{vkmeps}
	v_{k, m, \varepsilon} (x) :=  \varepsilon  |k|^{-m}e^{-x^2/2}  \cos(kx), \qquad  x\in \Reals^d. 
\end{equation}
Note that 
\[
	\calF v_{k, m, \varepsilon}  (\xi) =    \dfrac12  (2\pi)^{d/2}  \varepsilon  |k|^{-m} 
	 \left(e^{-(\xi -k)^2 /2} + e^{-(\xi  + k)^2/2 } \right),   \qquad  \xi\in \Reals^d. 
	\]
Similarly to  Remark \ref{Rem1}, we find that  the functions $v_{k, m, \varepsilon}$ satisfy  \eqref{eq:ass} for 
$
	\nu  =2,
$
 any $\sigma>\dfrac12$,  and 
  $N =   \varepsilon  |k|^{-m} N'(d,\sigma)$.
Then, for  any fixed $\sigma>\dfrac12$, 
$r>0$, integer $m>0$, and real  $\gamma_0, \gamma_1, \gamma_2>0$, we have that
\begin{equation}\label{Nnorms}
   N \leq \gamma_0, \qquad \| v_{k, m, \varepsilon}\|_{\calW^m(\Reals^d)} \leq \gamma_1,\qquad 
    \| v_{k, m, \varepsilon}\|_{\calH^m(\Reals^d)} \leq \gamma_2;
\end{equation}
 for all sufficiently small $\varepsilon>0$ and $|k|>1$;
 and, for fixed $\varepsilon>0$, 
  the following formulas hold as $|k|\rightarrow +\infty$:
 \begin{equation}\label{newestimates}
 \begin{aligned}
    \|\calF v_{k, m, \varepsilon}  \|_{\calL^\infty(B_r)} &= O\left(\exp(-\alpha|k|^2)\right), 
    \text{ for any $\alpha \in (0,\dfrac12)$,}\\
    \|v_{k, m, \varepsilon}\|_{\calL^\infty(\Reals^d)}  &= \varepsilon |k|^{-m}, 
  \qquad   \|v_{k, m, \varepsilon}\|_{\calL^2 (\Reals^d)}   = \Omega( |k|^{-m}).
 \end{aligned}
 \end{equation}
 It follows from \eqref{newestimates} that, for  fixed scaling parameter $\varepsilon$, 
  exponents $\beta_1,\beta_2>m/2$, and constants $c_1,c_2>0$,
 \begin{align}
	\|v_{k,m,\varepsilon}\|_{\calL^\infty(\mathbb{R}^d)} &> c_1 \left( \ln \left(3+ \left\|\mathcal{F} v_{k,m,\varepsilon} \right\|_{\calL^\infty(B_r)}^{-1} \right)\right)^{-\beta_1},
	\label{last1}
	 \\
		\|v_{k,m,\varepsilon}\|_{\calL^2(\mathbb{R}^d)}&> c_2 \left( \ln \left(3+ \left\|\mathcal{F} v_{k,m,\varepsilon} \right\|_{\calL^\infty(B_r)}^{-1} \right)\right)^{-\beta_2},\label{last2}
 \end{align}
 when $|k|$ is sufficiently large. 
 Condition 
    \eqref{Nnorms} and  instability estimates \eqref{last1}, \eqref{last2} 
show   nearly optimality of the exponent $\beta_1$
in  stability estimates  \eqref{eq_Theorem1}, \eqref{log-1}, \eqref{eq:cor1}
 and of the exponent $\beta_2$ in  stability estimates 
  \eqref{eq_Theorem2}, \eqref{log-2}, \eqref{eq:cor2}  with $s=0$,
  for the case when $\nu=2$.
  Namely, for this case,
   \[
    \beta_1^{\max}(2)  = \left(1- \dfrac{1}{\sqrt2}\right) (m-d)  \qquad  \text{ and } \qquad
    \beta_2^{\max}(2) = \left(1- \dfrac{1}{\sqrt2}\right)  m    \text{ (for $s=0$)}. 
    \]
  One can see that $m/2$, which is the infima of the exponents $\beta_1$ and $\beta_2$    in \eqref{last1} and \eqref{last2},  is substantially closer to   
   $\beta_1^{\max}(2)$ and $\beta_2^{\max}(2)$  than  the infima   of the exponents $\beta_1$ and $\beta_2$    in \eqref{instable:inf} and \eqref{instable:L2},  respectively.
   
   It is also important to note that the instability behaviour exhibiting by the functions 
   $v_{k,m,\varepsilon}$  defined in  \eqref{vkmeps} is impossible for  the compactly supported case, at least for sufficiently large $m$.  
   This is because $\beta_1^{\max}(1)$ and $\beta_2^{\max}(1)$ (for $s=0$)
    get bigger  than $m/2$  so    \eqref{last1} and \eqref{last2}  would contradict to Corollary \ref{Corollary}.

\section{Stability estimates for Problem \ref{Problem2}}\label{S:con}
%In this section, we give stability estimates for continuations   $\mathcal{C}_{R,n}$ defined according to \eqref{def_C}; see Lemma \ref{Lemma_C}, Theorem \ref{Theorem3}, and Corollary \ref{Corollary_C}. These estimates  require only  the assumptions of  \eqref{eq:ass}.
%and its proof consists in optimising the parameter $n$  in the bound of the lemma below. 

%
%\begin{Lemma}\label{Lemma_C}
%Let $r, \rho,R>0$ be such that $R\geq r$, $\rho \geq 4R/r$.
%Let  $v$ be such that $Q(r \rho/2)<+\infty$, where $Q(\cdot)$ is defined in \eqref{ass:eq}.  \red{Suppose that  $\|w - \calF v\|_{\calL^\infty(B_r)} \leq \delta$ for some
%function~$w$  and $\delta>0$.}  Then,     for  any $n \in \mathbb{N}$,  the following estimate holds:
%	\begin{equation*}\label{eq:Lemma_C}
%		\|\mathcal{F}v - \mathcal{C}_{R,n}[w] \|_{\calL^{\infty}(B_R)} 
%		\leq 2 \left(\dfrac{2R}{r}\right)^n \delta  + 4\,Q (r \rho/2) \left(\dfrac{2R}{r\rho}\right)^{n}.
%	\end{equation*}
%\end{Lemma}

\begin{Lemma}\label{Lemma_C}
 Let $v\in \calL^1(\Reals^d)$ and $\lambda, r, R>0$ be such that $r\leq R \leq \lambda/2$,  and $Q_v(\lambda)<+\infty$. 
 %\red{, where $Q_v(\cdot)$ is defined in \eqref{ass:eq}.}  %Suppose that  $\|w - \calF v\|_{\calL^\infty(B_r)} \leq \delta$ for some
%function~$w$  and $\delta>0$. 
  Then,     for  any
 $w\in \calL^\infty(B_r)$ and  $n \in \mathbb{N}$, the following estimate holds:
	\begin{equation*}\label{eq:Lemma_C}
		\|\mathcal{F}v - \mathcal{C}_{R,n}[w] \|_{\calL^{\infty}(B_R)} 
		\leq 2 \left(\dfrac{2R}{r}\right)^n \|w - \calF v\|_{\calL^\infty(B_r)}  + 4\,Q_v (\lambda) \left(\dfrac{R}{\lambda}\right)^{n}.
	\end{equation*}
\end{Lemma}

Lemma  \ref{Lemma_C} is proved  in Section \ref{S:Proof-Lemma}. 
%We proceed to our stability estimate for  Problem \ref{Problem2}.
Optimising the parameter $n$ in Lemma \ref{Lemma_C}, we obtain the following H\"older stability estimate for Problem \ref{Problem2}.

\begin{Theorem}\label{Theorem3}
 Let $v\in \calL^1(\Reals^d)$ and $\lambda, r, R>0$ be such that $r\leq R \leq \lambda/2$,  and $Q_v(\lambda)<+\infty$. 
 %\red{, where $Q_v(\cdot)$ is defined in \eqref{ass:eq}.}  
 Suppose that  $\|w - \calF v\|_{\calL^\infty(B_r)} \leq \delta$ for some
function~$w$ and  $0<\delta <  Q_v (\lambda)$. 
Then  the following estimate holds:
\begin{equation}\label{eq:Theorem3}
		\|\mathcal{F}v - \mathcal{C}_{R,n^*}w \|_{\calL^{\infty}(B_{R})} 
		\leq \dfrac{8R}{r} \left( \dfrac{Q_v(\lambda)}{\delta}\right)^{\tau(\lambda)} \delta,
	\end{equation}
	where 
		\begin{equation}\label{def:ntau}
		\begin{aligned}
		n^*:=  
		 \left\lceil \frac{ \ln  \left( \frac{Q_v(\lambda)}{\delta}\right) }{ \ln (2\lambda/r)}\right\rceil   
		 \qquad \text{and} \qquad 
		\tau(\lambda)   := \frac{\ln(2R/r)}{\ln (2\lambda/r) }.
		\end{aligned}
	\end{equation}
\end{Theorem}
%
%\noindent
%Note that $n^*\geq 0$ under the assumptions of Theorem \ref{Theorem3} since \red{
% \[ 
%  \ln \left( \dfrac{N }{2 \delta}  e^{r\sigma \rho/2} \right)  + \ln \rho  \geq  
%  \ln \left(\dfrac{(2\pi)^d}{2}\right) + \ln 4 >d \ln(2\pi).
%  \]}
\begin{Remark}
 Problem \ref{Problem2}   is a particular case of the problem of stable analytic continuation; see, for example,  Demanet, Townsend \cite{DT2019}, Lavrent'ev et al.  \cite[Chapter 3]{LRS1986},  Tuan~\cite{Tuan2000},
 and Vessella \cite{Vessella1999}. 
  In particular, 
 \cite[Theorem 1.2]{DT2019}  or  \cite[Theorem 1]{Vessella1999} 
 lead to a   H\"older stability  estimate similar to \eqref{eq:Theorem3}.
 In the  present work, we independently establish estimate~\eqref{eq:Theorem3}   mainly for the purpose to  give a simple explicit expression for the factor in front of the H\"older term $\delta^{1-\tau(\lambda)}$. 
 Besides, we derive our estimates  for  specific analytic functions which are the Fourier transforms of functions  satisfying \eqref{eq:ass}.
\end{Remark}
\begin{proof}[Proof of Theorem~\ref{Theorem3}]
By the assumptions, we have that
\[
 	\eta :=    \frac{ \ln  \left( \frac{Q_v(\lambda)}{\delta}\right) }{ \ln (2\lambda/r)}>0.
\] 
Note that $\eta$ is the solution of the equation
\[
	  \left(\dfrac{2R}{r}\right)^\eta \delta   =  Q_v (\lambda) \left(\dfrac{R}{\lambda}\right)^{\eta}.
\]
%and  $\delta = N    e^{r \sigma \rho} (3\rho)^{-\eta}$.
Using also that $R\geq r$, we get
\begin{align*}
		 \left(\dfrac{2R}{r}\right)^{\eta+1}   \delta        
		  = 
		  \dfrac{  2R}{r   }   Q_v(\lambda) \left(\dfrac{ R}{ \lambda}\right)^{\eta}
		 \geq    2\,  Q_v(\lambda) \left(\dfrac{ R}{ \lambda}\right)^{\eta}.
%	\\
%	   \left(\dfrac{4R}{r}\right)^{\mu} &=  	 \rho^{\tau(\rho) \mu }= 
%	   \left( \left(\dfrac{3}{2}\right)^d \dfrac{N}{\delta}  e^{\frac34 r \sigma \rho} \right)^{\tau(\rho)}.
\end{align*}
%where we used  the assumptions $\delta \leq (2\pi)^{-d} Ne^{r\sigma\rho/2} $ and  $ R\geq r$ to derive the last inequality.
By definition of $n^*$ and $\eta$,  we find that $\eta \leq n^* < \eta+1$.
Then, applying Lemma \ref{Lemma_C}, we obtain that
\begin{equation*}
	\begin{aligned}
		\|\mathcal{F}v - \mathcal{C}_{R,n^*}[w] \|_{\calL^{\infty}(B_R)} 
		&\leq  2 \left(\dfrac{2R}{r}\right)^{n^*} \delta  + 4\,Q_v (\lambda) \left(\dfrac{R}{\lambda}\right)^{n^*}
	\\	&\leq   2\left(\dfrac{2R}{r}\right)^{\eta+1} \delta  + 4\,Q_v (\lambda) \left(\dfrac{R}{\lambda}\right)^{\eta}		 \\ & \leq     4\left(\dfrac{2R}{r}\right)^{\eta+1} \delta .
%		 \leq 
%		  \dfrac{2R}{r} \left( \left(\dfrac{3}{2}\right)^d \dfrac{N }{ \delta} e^{\frac34 r\sigma \rho} \right)^{\tau(\rho)} 4^d \delta.
	\end{aligned}
\end{equation*}
Then, by the definitions of $\tau(\lambda)$ and $\eta$,  we get 
\[
		   \left(\dfrac{2R}{r}\right)^{\eta}  =  
		   \exp\left(\eta \ln\dfrac{2R}{r} \right)  =
%		   	  \left(\left(\dfrac{2\lambda}{r}\right)^{ \eta}\right)^{\tau(\lambda) } \delta= 
   \exp\left( \eta\, \tau(\lambda) \ln\dfrac{2\lambda}{r} \right)  =
	   \left(\dfrac{Q_v(\lambda)}{\delta}\right)^{\tau(\lambda)}.
\]
Combining the above formulas completes the proof.
\end{proof}
%%%%%%%%%%%%%%%%%%%%%%%%%%%%%%%%%%%%%%%%%%%%%%%%%%%%%%%%%%%%%

Theorem \ref{Theorem3} leads to the following stability estimate for the
 extrapolation   $\calC^*_{\tau,\delta} $ used in   Theorem \ref{Theorem2}. 
%Recall the definitions of $ L_{\tau}(\delta)$, $n_\tau(\delta)$ and $R_\tau(\delta)$ given in \eqref{def_L} and \eqref{def_R}.
\begin{Corollary}\label{Corollary_C}
 Let the assumptions of \eqref{eq:ass} and \eqref{eq:ass2}  hold for some $N,\sigma, r, \delta>0$ and $\nu\geq 1$.  
%Suppose that $\delta$ is such  that 
%\begin{equation*}
%	 L_{\tau}(\delta) =\left(\frac{(1-\tau) \ln (1+\delta^{-1})}{r\sigma}\right)^\tau \geq 4,
%\end{equation*}
%where  $L_{\tau}(\delta) $ is defined by  \eqref{def_L}. 
Then, for any $\tau \in [0, \nu^{-1}]$,
% $\alpha \in [0,1] $, 
 we have
 	\begin{equation*}
		\|\mathcal{F}v - \calC^*_{\tau,\delta} [w] \|_{\calL^{\infty}\left( B_{R_\tau(\delta)}\right)}
		\leq    8 L_{\tau}(\delta)   \left(\dfrac{N}{\delta}\right)^{\nu\tau(2-\tau) } \delta
=   8 L_{\tau}(\delta)  N^{1-\alpha}\delta^\alpha,
	%	\leq  16 L_{\alpha}(\delta)   N  \left(\dfrac{\delta}{N}\right)^{\alpha},
	\end{equation*}
%	where 
%	${\mathcal C}^*$  denotes ${\mathcal C}_{R,n}$ of \eqref{def_C} with 
%	\begin{equation}
%		\begin{aligned}
%		&R = R^* = \frac{r}{2}\, L_{r,\sigma,\tau},\\
%		&n = n^* = \left\lceil\frac{\tau \ln \left( \frac{2}{\pi}N\delta^{-1}(1+\delta^{-1})^{1-\tau} \right)}{ \ln L_{r,\sigma,\tau}}\right\rceil.
%		\end{aligned}
%	\end{equation}
where $L_{\tau}(\delta)$ and $R_\tau (\delta)$  are defined in \eqref{def_L} and \eqref{def_R} and $\alpha= \alpha(\tau) := 1- \nu \tau(2-\tau)$. In addition, the exponent $\alpha$
is positive if and only if   $0\leq \tau< 1 - \sqrt{1- \nu^{-1}}\leq \nu^{-1}$.
\end{Corollary}
%Note that, for 
%$\alpha:= 1- \nu \tau (2-\tau)$,  
%\[
%	  \left(\dfrac{N}{\delta}\right)^{\nu\tau(2-\tau) } \delta  =  
%	  N^{1-\alpha}\delta^\alpha.
%\]
% The exponent $\alpha$ is positive if and only if $\tau<?<\nu^{-1}$.
%%%%%%%%%%%%%%%%%%%%%%%%%%%%%%%%%%%%%%%%%%%%%%%%%%%%%%%%%%%%%%%%%%%%%%%%%%%%%%%%%%%%%%%%%%%%%%%%%%%%%%%%%%%%
%%%%%%%%%%%%%%%%%%%%%%%%%%%%%%%%%%%%%%%%%%%%%%%%%%%%%%%%%%%%%%%%%%%%%%%%%%%%%%%%%%%%%%%%%%%%%%%%%%%%%%%%%%%%
\begin{proof} 
%By assumptions, we find that 
%\begin{equation*}%\label{cond_L}
% L_{r,\sigma,\tau}(\delta) = \left(\frac{(1-\tau) \ln (1+\delta^{-1})}{r\sigma}\right)^\tau.
%\end{equation*}
%%This implies  $ \delta <1$.
%First, we consider the case    $\tau=0$,  for which  $n_\tau(\delta) =0$, by  \eqref{def_R}. 
%\red{Using \eqref{eq:N-delta} and recalling from \eqref{eq:ass2} and \eqref{def_L}}   that
%$\delta<N$ and
% $L_{\tau}(\delta) \geq 1$, we get 
%\begin{align*}
%	\|\calF v - \calC^*_{\tau,\delta}w\|_{\calL^{\infty}\left([-R_\tau(\delta), R_\tau(\delta)]^d\right)} 
%	\leq \|\calF v\|_{\calL^{\infty}(\Reals^d)} \leq N 
%	 \leq  \left(\dfrac{16}{3}\right)^d N  \left(\dfrac{\delta}{N}\right)^{(1-\tau)^2}L_{\tau}(\delta).
%\end{align*}

First, we consider the case $L_{\tau}(\delta) = 1$.
Then,  \eqref{def_R} and  \eqref{def_C} imply that $R_\tau(\delta) = r$ and 
   $\calC^*_{\tau,\delta}[w] \equiv w$.
%
%
%First, we consider the case $L_{\tau}(\delta) = 1$, for which  $R_\tau(\delta) = r$. 
Recalling from \eqref{eq:ass2} that $\delta<N$  , we find that
\begin{align*}
	\|\calF v - \calC^*_{\tau,\delta}[w]\|_{\calL^{\infty}\left( B_{R_\tau(\delta)}\right)}
	=  \|\calF v  -w\|_{\calL^\infty(B_r)} \leq  \delta \leq   8 L_{\tau}(\delta)   \left(\dfrac{N}{\delta}\right)^{\nu\tau(2-\tau) } \delta.
\end{align*}
Next, suppose that 
\[	
L_{\tau}(\delta) = \dfrac12\left(\dfrac{(1-\tau) \ln \frac{N}{\delta}}{\sigma r^{\nu}}\right)^\tau>1.
\]
This is only possible when $\tau\neq 0$. Let
	\[
		\lambda := r(2L_{\tau}(\delta))^{\frac{1}{\nu \tau}}.
	\]
	Then, from  \eqref{eq:ass},  we get 
	\[
	Q_v(\lambda)
	\leq N \exp( \sigma \lambda^{\nu}) =  \delta \left(\dfrac{N}{\delta}\right)^{2-\tau}.
	\]
	 In addition, by the
	assumptions, 
	\begin{align*}
		R_\tau (\delta)\geq r  \qquad \text{and} \qquad \lambda \geq   r \left( 2L_{\tau}(\delta)\right) =   2R_\tau(\delta).  
			\end{align*} 
			Observe  that $n^*$  defined in \eqref{def:ntau} for   $\lambda=  r(2L_{\tau}(\delta))^{\frac{1}{\nu \tau}}$ 			 coincides with $n_\tau(\delta)$ defined in \eqref{def_R}.
	 Then, applying Theorem \ref{Theorem3}, we get that
		\begin{align*}
		\|\mathcal{F}v - \calC^*_{\tau,\delta} [w] \|_{\calL^{\infty}\left(  B_{R_\tau(\delta)}\right)} 
		\leq  8 L_{\tau}(\delta)  \left( \left(\dfrac{N}{\delta}\right)^{2-\tau}\right)^{\tau(\lambda)} \delta,
%		\\
%		&\leq
%		 16 L_{\tau}(\delta)  \left( \dfrac{N \exp(\sigma (r\rho)^{\nu})}{\delta}\right)^{\tau(\rho)}
		\end{align*}
		 where $\tau(\lambda)$ is defined in \eqref{def:ntau}.  
		 Note that  $\tau(\lambda)$ is  different from $\tau$. However, we can replace
$\tau(\lambda)$ by $\nu \tau$ in the estimate above
since  $ \delta < N$ and
 	\[
	\tau(\lambda)  = \frac{\ln(2R_\tau(\delta)/r)}{\ln (2\lambda/r)} = 
	\frac{\ln (2L_{\tau}(\delta))}
	{\ln 2 + \frac{1}{\nu \tau} \ln  (2L_{\tau}(\delta))} \leq \nu \tau.
	\]
%	Observe that $(1+t)^{\tau-\tau^2} \le 1+ t^{1-\tau}$ for any $t\geq 0$ since it is true for $t=0$ and
%	\[ 
%	\left((1+t)^{\tau-\tau^2}\right)' = 
%	\tau(1-\tau)(1+t)^{\tau-\tau^2 -1} \leq (1-\tau)(1+t)^{-\tau} \leq \left(1+ t^{1-\tau}\right)'.
%	\]
%	Recalling $\delta \leq N$, we get that 
%		\begin{align*}
%		(N (1+\delta^{-1})^{1-\tau})^{\tau} 
%		\leq  N^{\tau} (1+N)^{\tau-\tau^2}\delta^{\tau^2 -\tau} \leq
%		  (N + N^{\tau}) \delta^{\tau^2 -\tau}.
%		%\delta^{\tau^2 -\tau} (1+N)^{(2-\tau) \tau} \leq \delta^{\tau^2 -\tau} (1+N).
%%		  &\leq  \delta^{\tau -1} N^{\tau}
%%		  		  \left(N(N+1) \right)^{1-\tau}
%%		  		  \\
%%				&\leq 
%%				\delta^{\tau-2} \left(\frac{N}{(2\pi)^d} \right)^{\tau}
%%				 \left(\frac{N} {(2\pi)^d}\left(2+ \frac{N} {(2\pi)^d}\right)\right)^{1-\tau}
%%				 \\
%%				& \leq \delta^{\tau-2} \left(1+ \frac{2N}{(2\pi)^d} \right).
%	\end{align*}
	The required bound follows.
\end{proof}
%%%%%%%%%%%%%%%%%%%%%%%%%%%%%%%%%%%%%%%%%%%%%%%%%%%%%%%%%%%%%%%%%%%%%%%%%%%%%%%%%%%%%%%%%%%%%%%%%%%%%%%%%%%%
%%%%%%%%%%%%%%%%%%%%%%%%%%%%%%%%%%%%%%%%%%%%%%%%%%%%%%%%%%%%%%%%%%%%%%%%%%%%%%%%%%%%%%%%%%%% 
%By assumptions, we find that 
%\begin{equation*}%\label{cond_L}
% L_{r,\sigma,\tau}(\delta) = \left(\frac{(1-\tau) \ln (1+\delta^{-1})}{r\sigma}\right)^\tau.
%\end{equation*}
%%This implies  $ \delta <1$.
%%%%%%%%%%%%%%%%%%%%%%%%%%%%%%%%%%%%%%%%%%%%%%%%%%%%%%%%%%%%%%%%%%%%%%%%%%%%%%%%%%%%%%%%%%%%%%%%%%%%%%%%%%%%
%%%%%%%%%%%%%%%%%%%%%%%%%%%%%%%%%%%%%%%%%%%%%%%%%%%%%%%%%%%%%%%%%%%%%%%%%%%%%%%%%%%%%%%%%%%%%%%%%%%%%%%%%%%%
%%%%%%%%%%%%%%%%%%%%%%%%%%%%%%%%%%%%%%%%%%%%%%%%%%%%%%%%%%%%%%%%%%%%%%%%%%%%%%%%%%%%%%%%%%%%%%%%%%%%%%%%%%%%
%%%%%%%%%%%%%%%%%%%%%%%%%%%%%%%%%%%%%%%%%%%%%%%%%%%%%%%%%%%%%%%%%%%%%%%%%%%%%%%%%%%%%%%%%%%%%%%%%%%%%%%%%%%%

\section{ Proofs of  Theorem \ref{Theorem1} and Theorem \ref{Theorem2}}\label{S:proofs}

In this section,  we prove Theorem \ref{Theorem1}  and Theorem \ref{Theorem2}. Their proofs  are very similar. Starting from the inverse Fourier transform formula
\begin{equation*}
	v(x)  = \int\limits_{\mathbb{R}^d} e^{-i\xi x}\mathcal{F}{v}(\xi) d\xi, \ \ \ x\in \mathbb{R}^d, 
\end{equation*}
we analyse  the contributions of the two regions $B_{R_\tau(\delta)}$
 and $\Reals^d \setminus B_{R_\tau(\delta)}$. 
 For the first region, we apply Corollary \ref{Corollary_C}. For the second region 
 we use  the smoothness assumptions $\|v\|_{\calW^m(\Reals^d)}\leq \gamma_1$ 
 or  $\|v\|_{\calH^m(\Reals^d)}\leq \gamma_2$.

 Note that  $\alpha$ in estimates \eqref{eq_Theorem1}, \eqref{eq_Theorem2} is the same as in Corollary  \ref{Corollary_C}.
Indeed, as stated in Theorem \ref{Theorem1}  and Theorem \ref{Theorem2}, 
for a given $\alpha \in [0,1]$, we define
\[
	\tau = \tau(\alpha):= 1 -\sqrt{1-(1-\alpha)\nu^{-1}}. 
\]
Then, observe that $0\leq \tau \leq 1 - \sqrt{1- \nu^{-1}}\leq \nu^{-1}$ and 
\[
	1 - \nu \tau(2-\tau) = \alpha,
\]
as in Corollary  \ref{Corollary_C}.
%This observation shows that the parameter $\alpha$ defined Corollary \ref{Corollary_C}
%coincides  with the original $\alpha$.
% 

Recall also the definitions of $L_{\tau}(\delta)$, $R_\tau(\delta)$ and $c(d)$  from \eqref{def_L}, \eqref{def_R},  and \eqref{def_c}, respectively. 
It is straightforward to check that \eqref{log-1} and \eqref{log-2}  follow
from  \eqref{eq_Theorem1}  and 
\eqref{eq_Theorem2}, respectively. Indeed,  for  any $N,\sigma,r >0$,    $\nu\geq 1$,   
$\tau \in (0, 1-\sqrt{1-\nu^{-1}})$, we have that 
\[
	C_{1} \left( \ln (3+\delta^{-1})\right)^{\tau} \leq L_{\tau}(\delta) \leq C_{2} \left( \ln (3+\delta^{-1})\right)^{\tau}, \qquad \text{for   $0<\delta<N$,}
\]
where $C_{1}=C_1(N,\sigma,\nu, r, \tau)>0$ and $C_{1}=C_1(N,\sigma,\nu, r, \tau)>0$. 
Then, the second terms of the right-hand side on the estimates \eqref{eq_Theorem1}  and 
\eqref{eq_Theorem2} dominates the first terms as $\delta\rightarrow 0$. Observing also that 
$\ln (3+\delta^{-1}) \geq 1$ for all $\delta>0$, we can find some suitable constants $c_1$ and  $c_2$ such that  estimates  \eqref{log-1} and \eqref{log-2} always hold  (given the assumptions)
 with $\beta_1 = (m-d)\tau$ and $\beta_2=(m-s)\tau$. 
Thus, to complete the proofs of Theorems \ref{Theorem1}
and  \ref{Theorem2},  it  remains to establish stability estimates  \eqref{eq_Theorem1} and  \eqref{eq_Theorem2}. 

\subsection{ Proof of estimate  (\ref{eq_Theorem1})}\label{S:2.1}

Observe that
\begin{equation*}\label{4.2}
	\begin{aligned}
		\|v - \mathcal{F}^{-1} \calC^*_{\tau,\delta} w\|_{\calL^{\infty}(\mathbb{R}^d)}  \leq 
		\sup\limits_{x\in \mathbb{R}^d}\int\limits_{\mathbb{R}^d}  \left|
		e^{-i\xi x}\left(\mathcal{F}{v}(\xi) -  \calC^*_{\tau,\delta} w(\xi)\right)\right| d\xi 
	 =  I_1 + I_2.
	\end{aligned}
\end{equation*}
 where
\begin{equation*}
	\begin{aligned}
		I_1 &:= \int\limits_{B_{R_\tau(\delta)}} |\mathcal{F}{v}(\xi) -  \calC^*_{\tau,\delta} w(\xi)| d\xi, \\ 
		I_2 &:= \int\limits_{\mathbb{R}^d \setminus B_{R_\tau(\delta)}} |\mathcal{F}{v}(\xi) -  \calC^*_{\tau,\delta} w(\xi)| d\xi.
	\end{aligned}
\end{equation*}
%	By assumptions, we find that
%	\begin{equation*}
%		I_1 = \int\limits_{B_r} |\mathcal{F}{v}(\xi) - w(\xi)| d\xi \leq  \int\limits_{B_r} \delta\, d\xi = \frac{c   r^d}{d} \delta.
%	\end{equation*}
		Using  Corollary \ref{Corollary_C}, we get that
	\begin{equation*} 
	 	\begin{aligned}
	 	I_1 \leq   \int \limits_{B_{R_\tau(\delta)}} \left\|\mathcal{F}{v} - \calC^*_{\tau,\delta} w\right\|_{\calL^\infty(B_{R_\tau(\delta)})} d\xi 
	 	&\leq    \int\limits_{B_{R_\tau(\delta)}}    8 L_{\tau}(\delta)   N^{1-\alpha} \delta^{\alpha}\,  d \xi\\
	 	&= \dfrac{8 c(d)}{d}  N^{1-\alpha}  r^d   \left({L}_{\tau}(\delta)\right)^{d+1} 
		   \delta^{\alpha}.
	 	\end{aligned}
	 \end{equation*}
	 	 
Next,  since $v\in \calW^m(\mathbb{R}^d)$, we have that 
\begin{equation*}
	 |\xi|^m |\mathcal{F}{v}(\xi)| \leq (1+|\xi|^2)^{m/2} |\mathcal{F}{v}(\xi)| \leq  \|v\|_{\calW^m(\Reals^d)}.
\end{equation*}
Thus, we can bound
\begin{equation*}
	\begin{aligned}
	I_2&= \int\limits_{{\mathbb{R}^d \setminus B_{R_\tau(\delta)}}} |\calF v(\xi)| d\xi  \leq c(d) \int\limits_{R_\tau(\delta)}\limits^{+\infty} \frac{\|v\|_{\calW^m(\Reals^d)}}{t^{m-d-1}}dt \\ &=  
	\dfrac{c(d)}{m-d}  \|v\|_{\calW^m(\Reals^d)}   (R_\tau(\delta))^{-m+d} \\
	&=
	\dfrac{c(d)}{m-d}     \|v\|_{\calW^m(\Reals^d)}  	r^{-m+d} \left({L}_{\tau}(\delta)\right)^{-m+d}.
	\end{aligned}
\end{equation*}
Combining the above bounds for $I_1$ and $I_2$  completes the proof of  \ \eqref{eq_Theorem1}. 
%%%%%%%%%%%%%%%%%%%%%%%%%%%%%%%%%%%%%%%%%%%%%%%%%%%%%%%%%%%%%%%%%%%%%%%%%%%%%%%%%%%%%%%%%%%%%%%%%%%%%%%%%%%%
%%%%%%%%%%%%%%%%%%%%%%%%%%%%%%%%%%%%%%%%%%%%%%%%%%%%%%%%%%%%%%%%%%%%%%%%%%%%%%%%%%%%%%%%%%%%%%%%%%%%%%%%%%%%
%%%%%%%%%%%%%%%%%%%%%%%%%%%%%%%%%%%%%%%%%%%%%%%%%%%%%%%%%%%%%%%%%%%%%%%%%%%%%%%%%%%%%%%%%%%%%%%%%%%%%%%%%%%%
%%%%%%%%%%%%%%%%%%%%%%%%%%%%%%%%%%%%%%%%%%%%%%%%%%%%%%%%%%%%%%%%%%%%%%%%%%%%%%%%%%%%%%%%%%%%%%%%%%%%%%%%%%%%
%%%%%%%%%%%%%%%%%%%%%%%%%%%%%%%%%%%%%%%%%%%%%%%%%%%%%%%%%%%%%%%%%%%%%%%%%%%%%%%%%%%%%%%%%%%%%%%%%%%%%%%%%%%%

\subsection{Proof of  estimate (\ref{eq_Theorem2}) }\label{S:2.2}

The Parseval-Plancherel identity states that  
		\begin{equation}\label{Parseval}
		 \|u\|_{\calL^2(\mathbb{R}^d)} = (2\pi)^{\frac d 2}\|\mathcal{F} u \|_{\calL^2(\mathbb{R}^d)} = (2\pi)^{-\frac d 2}\|\mathcal{F}^{-1} u \|_{\calL^2(\mathbb{R}^d)}. 
		\end{equation}
Thus, we get that 
\begin{equation*}%\label{5.2}
	\begin{aligned}
	\|v - \mathcal{F}^{-1} \calC^*_{\tau,\delta} w\|_{\calH^s(\mathbb{R}^d)} = 
	(2\pi)^{\frac d 2} 
	\left\|(1+|\xi|^2)^{\frac{s}{2}}(\mathcal{F} v -  \calC^*_{\tau,\delta} w)\right\|_{\calL^{2}(\mathbb{R}^d)} 
	\leq (2\pi)^{\frac d 2} (\tilde{I}_1+\tilde{I}_2),
	\end{aligned} 
\end{equation*}
where
\begin{equation*}
	\begin{aligned}
		\tilde{I}_1 &:= \left(\int\limits_{B_{R_\tau(\delta)}} (1+|\xi|^2)^{s}|\mathcal{F}{v}(\xi) -
		\calC^*_{\tau,\delta} w(\xi)|^2 d\xi \right)^{1/2}, \\ 
		\tilde{I}_2 &:= \left(\int\limits_{\mathbb{R}^d \setminus B_{R_\tau(\delta)}} (1+|\xi|^2)^{s} |\mathcal{F}{v}(\xi) - \calC^*_{\tau,\delta} w(\xi)|^2 d\xi \right)^{1/2}.
	\end{aligned}
\end{equation*}
Using  Corollary \ref{Corollary_C}, we  get that
	 \begin{equation*}	
	 	\begin{aligned}
	 	\tilde{I}_1 &\leq   \left( \int \limits_{B_{R_\tau(\delta)}} 
	 	(1+|\xi|^2)^{s}\left\|\mathcal{F}{v} - \calC^*_{\tau,\delta}w\right\|^2_{\calL^\infty(B_{R_\tau(\delta)})} d\xi  \right)^{1/2} \\
	 	&\leq  8  N^{1-\alpha}   \left(c(d) \int\limits_{0}^{R_\tau(\delta)} (1+t^2)^{s} t^{d-1} dt   \right)^{1/2}       L_{\tau}(\delta) \delta^{\alpha}. 
	 %	\left( \int \limits_{B_{R(\tau,\delta)}\setminus B_r} 
	 	%(1+|\xi|^2)^{s} d\xi  \right)^{1/2}.
	 	\end{aligned}
	 \end{equation*}
Applying  \eqref{Parseval} and recalling  that  $v\in \calH^m(\mathbb{R}^d)$, we find that 
\begin{equation*}
	 \int\limits_{\Reals^d \setminus B_{R_\tau(\delta)}} (1+|\xi|^2)^{s} |\mathcal{F} v(\xi)|^2 d\xi
		\leq  \left\|\frac{(1+|\xi|^2)^{\frac{m}{2}} \mathcal{F} v }{(R_\tau(\delta))^{m-s}} \right\|_{L^2(\mathbb{R}^d \setminus B_{R(\tau,\delta)})}^2
		\leq \frac{(2\pi)^{-d} \|v\|_{\calH^m(\Reals^d)}^2}{(R_\tau(\delta))^{2(m-s)}}.
\end{equation*}
	Thus, we can bound
	\begin{equation*}
		\begin{aligned}
			\tilde{I}_2 &\leq  \left(\int\limits_{\mathbb{R}^d \setminus B_{R_\tau(\delta)}} (1+|\xi|^2)^{s} |\mathcal{F}{v}(\xi)|^2  d\xi \right)^{1/2}
			 \leq \left\|\frac{(1+|\xi|^2)^{\frac{m}{2}} \mathcal{F} v }{(R_\tau(\delta))^{m-s}} \right\|_{L^2(\mathbb{R}^d \setminus B_{R_\tau(\delta)})} \\
			&\leq \frac{(2\pi)^{-\frac d 2}  \|v\|_{\calH^m(\Reals^d)}}{(R_\tau(\delta))^{m-s}}
			=  (2\pi)^{- \frac d 2}  \|v\|_{\calH^m(\Reals^d)} r^{-m+s}\left({L}_{\tau}(\delta)\right)^{-m+s}.
		\end{aligned}
	\end{equation*}
	Combining the above bounds for $\tilde{I}_1$ and $\tilde{I}_2$ completes the proof of
	 \eqref{eq_Theorem2}.
%%%%%%%%%%%%%%%%%%%%%%%%%%%%%%%%%%%%%%%%%%%%%%%%%%%%%%%%%%%%%%%%%%%%%%%%%%%%%%%%%%%%%%%%%%%%%%%%%%%%%%%%%%%%

%%%%%%%%%%%%%%%%%%%%%%%%%%%%%%%%%%%%%%%%%%%%%%%%%%%%%%%%%%%%%%%%%%%%%%%%%%%%%%%%%%%%%%%%%%%%%%%%%%%%%%%%%%%%
%%%%%%%%%%%%%%%%%%%%%%%%%%%%%%%%%%%%%%%%%%%%%%%%%%%%%%%%%%%%%%%%%%%%%%%%%%%%%%%%%%%%%%%%%%%%%%%%%%%%%%%%%%%%
%%%%%%%%%%%%%%%%%%%%%%%%%%%%%%%%%%%%%%%%%%%%%%%%%%%%%%%%%%%%%%%%%%%%%%%%%%%%%%%%%%%%%%%%%%%%%%%%%%%%%%%%%%%%
%%%%%%%%%%%%%%%%%%%%%%%%%%%%%%%%%%%%%%%%%%%%%%%%%%%%%%%%%%%%%%%%%%%%%%%%%%%%%%%%%%%%%%%%%%%%%%%%%%%%%%%%%%%%
%%%%%%%%%%%%%%%%%%%%%%%%%%%%%%%%%%%%%%%%%%%%%%%%%%%%%%%%%%%%%%%%%%%%%%%%%%%%%%%%%%%%%%%%%%%%%%%%%%%%%%%%%%%%
%%%%%%%%%%%%%%%%%%%%%%%%%%%%%%%%%%%%%%%%%%%%%%%%%%%%%%%%%%%%%%%%%%%%%%%%%%%%%%%%%%%%%%%%%%%%%%%%%%%%%%%%%%%%
%%%%%%%%%%%%%%%%%%%%%%%%%%%%%%%%%%%%%%%%%%%%%%%%%%%%%%%%%%%%%%%%%%%%%%%%%%%%%%%%%%%%%%%%%%%%%%%%%%%%%%%%%%%%
%%%%%%%%%%%%%%%%%%%%%%%%%%%%%%%%%%%%%%%%%%%%%%%%%%%%%%%%%%%%%%%%%%%%%%%%%%%%%%%%%%%%%%%%%%%%%%%%%%%%%%%%%%%%

\section{Proof of  Lemma \ref{Lemma_C}}\label{S:Proof-Lemma}

%%%%%%%%%%%%%%%%%%%%%%%%%%%%%%%%%%%%%%%%%%%%%%%
%%%%%%%%%%%%%%%%%%%%%%%%%%%%%%%%%%%%%%%%%%%%%%%%%%%%%%%%%%%%%%%%%%%%%%%%%%%%%%%%%%%%%%%%%%%%%%%%%%%%%%%%%%%%

To prove  Lemma \ref{Lemma_C}, we need a bound for the error term in approximations of holomorphic functions by truncated series of Chebyshev polynomials stated in the following lemma.  For completeness purposes, we include a proof of this bound. 

\begin{Lemma}\label{Lemma_1D}
 Suppose that $f(z)$ is a holomorphic function in the ellipse 
 \begin{equation}\label{D_rho}
 	D(\rho) := \left\{\cos z \st z\in \Complexes  \text{ and }  |\Im z| < \ln \rho \right\} 
 \end{equation}
 for some $\rho >2$ and $\sup\limits_{z \in D(\rho)} |f(z)| \leq M_\rho < +\infty$ for some $M_\rho >0$. 
 Then, 
 \begin{equation*}
 	\left\|f - \sum \limits_{k=0} \limits^{n-1} b_k T_k\right\|_{L^{\infty}([-\rho',\rho'])} \leq  
 	2 M_\rho   \left(1 - \dfrac{2\rho'}{\rho}\right)^{-1} \left(\dfrac{2\rho'}{\rho}\right)^n, 
 \end{equation*}
 for any $n \in \mathbb{N}$ and any $\rho' \in [1,\rho/2)$, where $(T_k)_{k \in \mathbb{N}}$ are the Chebyshev polynomials and 
 \begin{equation}\label{def_b} 
 		b_k	 := 
 		\begin{cases}
  \displaystyle 		 \dfrac{1}{\pi}\int\limits_{-1}\limits^1 \frac{f(t)}{\sqrt{1-t^2}} dt, &\text{if $k=0$,}
 		\\
 		\displaystyle \dfrac{2}{\pi}\int\limits_{-1}\limits^1 \frac{f(t)T_k(t)}{\sqrt{1-t^2}}  dt,& \text{otherwise.} 
 		\end{cases}
 \end{equation}
\end{Lemma}

\begin{proof} 
First of all, we note that the condition $\rho' \in [1,\rho/2)$ 
ensures that interval $[-\rho',\rho']$ (of the real axe) lies in the ellipse $D(\rho)$. Indeed, 
$$D(\rho) \cap \mathbb{R} = \left(-\dfrac{\rho+\rho^{-1}}{2},\dfrac{\rho+\rho^{-1}}{2}\right).$$ 
Note also that
	\begin{equation}\label{eq:z-cos}
			 |\Im \zeta| < \rho/2 \qquad  \text{for }  \zeta \in D(\rho).
	\end{equation}	

	 Let  $g(z) := f(\cos z)$. Note that
	$g(z)$ 
	is  an even $2\pi$-periodic holomorphic function in the stripe $|\Im z| < \ln \rho$ and, 
	for all $k \in \Naturals$, 
	\begin{equation*}
		 \int\limits_{0}\limits^{2\pi} e^{ik\varphi} g(\varphi) d\varphi  = \int\limits_{0}\limits^{2\pi} e^{-ik\varphi} g(\varphi) d\varphi
		=		2 \int\limits_{-1} \limits^1 \frac{f(t)T_k(t)}{\sqrt{1-t^2}} dt, % \ \ \ \ k \in \mathbb{N}.
	\end{equation*}
	Hence,  by the Cauchy integral theorem,  
	we get that
	\begin{equation}\label{eq:g}
		g(z) = \sum\limits_{k=0}\limits^{\infty} b_k \cos kz, \ \ \text{ for } 
	 	|\Im z| < \ln \rho,
	\end{equation}
	where
	\begin{equation*}\label{est_cn1}
		\begin{aligned}
		|b_k| = \left|\frac{1}{\pi} \int\limits_0\limits^{2\pi} e^{ik \varphi} g(\varphi) d\varphi  \right|
		= \left| \frac{1}{\pi}
		\int\limits_{0 + i\ln \rho }\limits^{2\pi + i\ln \rho} e^{ikz} g(z) dz  \right| \leq \\ 
		\leq \frac{1}{\pi}\int\limits_0\limits^{2\pi} e^{-k \ln\rho} |g(t)| dt  \leq  2 M_\rho \rho^{-k}, \ \ \    k \in \mathbb{N}.	
		\end{aligned}
	\end{equation*}
 Using \eqref{eq:g} and recalling that  $T_k(t):= \cos(k\operatorname{arccos}(t))$ for $|t|\leq 1$,  we get that 
\begin{equation*}\label{expansion_f}
		f(z) = \sum\limits_{k=0}\limits^{\infty}   b_k  T_k(z), \qquad
	 	z \in D(\rho).
	\end{equation*}
 Observe that 
if $|t| \leq 1$ then $|T_k(t)| \leq 1$, otherwise 
\begin{equation}\label{bound_T}
  \begin{aligned}
 	|T_k(t)| = |\cosh (k\operatorname{arccosh}(t))| = \dfrac{1}{2}|(t-\sqrt{t^2-1})^k + (t+\sqrt{t^2-1})^k|
 	\leq (2|t|)^k.
% 	\\ 
 %	&\leq \dfrac12 \left( \dfrac{1}{|x|+\sqrt{x^2-1}} + |x|+\sqrt{x^2-1}\right)^k
   \end{aligned}
\end{equation}
%	
%	
%\red{Combining the estimates above and using the well-known bound} \mi{give refs} 
%\begin{equation}\label{bound_T}
%\|T_k\|_{L^{\infty}[-\lambda;\lambda]} \leq (2\lambda)^k, \qquad \text{for any } \lambda\geq 1,
%\end{equation}
 Combining the estimates above, we get that, for any $t\in [-\rho',\rho']$ and  $n \in \mathbb{N}$,
\begin{equation*}\label{est_f_out}
	\begin{aligned}
	\left|f(t) - \sum\limits_{k=0}\limits^{n-1} b_k T_k(t)\right| &\leq 
	\sum\limits_{k=n}\limits^{\infty} |b_k T_k(t)| 
	\\&\leq 
	2 M_\rho \sum\limits_{k=n}\limits^{\infty}  \left(\dfrac{2\rho'}{\rho}\right)^k  =
2 M_\rho   \left(1 - \dfrac{2\rho'}{\rho}\right)^{-1} \left(\dfrac{2\rho'}{\rho}\right)^n.
	\end{aligned}
\end{equation*}
 This completes the proof of Lemma \ref{Lemma_1D}. 
\end{proof}

%%%
%%%%%%%%%%%%%%%%%%%%%%%%%%%%%%%%%%%%%%%%%%%%%%%%%%%%%%%%%%%%%%%%%%%%%%%%%%%%%%%%%%%%%%%%%%%%%%%%%%%%%%%%%%%%
%%%%%%%%%%%%%%%%%%%%%%%%%%%%%%%%%%%%%%%%%%%%%%%%%%%%%%%%%%%%%%%%%%%%%%%%%%%%%%%%%%%%%%%%%%%%%%%%%%%%%%%%%%%%

Now we are ready to prove  Lemma \ref{Lemma_C}.
For $\theta \in S^{d-1}$,  
consider functions $f_{r,\theta} : \mathbb{R} \rightarrow \mathbb{C}$ defined by
	\begin{equation*}\label{def_f_theta}
		f_{r,\theta} (s)  := \mathcal{F}v \left(sr\, \theta\right) =
		 \dfrac{1}{(2\pi)^d} \int\limits_{\mathbb{R}^d} e^{i sr \theta x } v(x) dx, \ \ s\in \mathbb{R}.
	\end{equation*}
	Provided  $Q_v(\lambda) < +\infty$,  we have that $f_{r,\theta}$  admits a holomorphic extension to
	 the ellipse  $D(\rho)$   defined by \eqref{D_rho} with $\rho:= 2\lambda/r$.	  Furthermore,  using \eqref{eq:z-cos}, we get that 
	   	\[
		|f_{r,\theta} (\zeta )|\leq   \dfrac{1}{(2\pi)^d} \int\limits_{\mathbb{R}^d}  e^{   r  |\Im\zeta| \cdot |x|   }| v(x)| dx     \leq   Q_v(\lambda), \qquad \text{  for $\zeta \in D(\rho)$.}
 	\]
% 	For the above, we used  the fact that
%	\begin{equation*}
%			\text{ $|\Im (\cos z)| \leq \rho/2$ \  for \  $|\Im  z| \leq \ln \rho$. }
%	\end{equation*}	
% 	
% 	
	Applying Lemma \ref{Lemma_1D} and taking into account that 
	$\rho = 2\lambda/r$ and  $r\leq R\leq \lambda/2$, we find  that   
%	and 
%	 estimate \eqref{bound_T} with	$\lambda := R/r$, we get that 
	\begin{equation*}%\label{f_theta_eq}
		\left\|f_{r,\theta} - \sum \limits_{k=0} \limits^{n-1} a_k(\theta) T_k\right\|_{\calL^{\infty}([-R/r,R/r])} \leq  
 	2 \,Q_v(\lambda)   \left(1 - \dfrac{R}{\lambda}\right)^{-1} \left(\dfrac{R}{\lambda}\right)^n.
	\end{equation*}
	%where, for $k\in \mathbb{N}$, the coefficients $b_k(\theta)$ denote  $b_k$ of (\ref{def_b}) for the functions $f_{r,\theta}$. 
					% definitions \eqref{def_C}, \eqref{def_b}, \eqref{def_f_theta}, estimate \eqref{f_theta_eq} and
			 It follows that
	\begin{equation}  \label{eq_Fw1}
			\left\|\mathcal{F}v - \mathcal{C}_{R,n}\left[\mathcal{F}v\right] \right\|_{\calL^{\infty}(B_R)} 
		\leq  4 \,Q_v(\lambda)   \left(\dfrac{R}{\lambda}\right)^{n}.
	\end{equation}
		We note  that 
	\begin{equation*}\label{linear_C}
			\mathcal{C}_{R,n}[w] - \mathcal{C}_{R,n}\left[\mathcal{F} v\right]  =
			\mathcal{C}_{R,n} \left[ w - \mathcal{F} v \right].
	\end{equation*}
%		 Using the following well-known bounds \mi{give refs}
%%\begin{equation*}
%%	\begin{aligned}
%%	T_k(x) = \cos(k \text{arccos}\,x), \ \ \	 x\in \mathbb{R},\ |x|\leq 1,\\
%%	T_k(x) = \frac{(x-\sqrt{x^2-1})^k + (x+\sqrt{x^2-1})^k}{2}, \ \ \	 x\in \mathbb{R},\ |x|>1,
%%\end{aligned}
%%\end{equation*}
%%we get that
%\begin{equation}\label{est_Tn}
%	\begin{aligned}
%	|T_k(x)| &\leq 2^{k} \lambda^k, \ \ \ \ x\in [-\lambda,\lambda],\\
%	|T_k(x)| &\leq 1,  \ \ \ \ x\in [-1,1].
%	\end{aligned}
%\end{equation}
Observe  that
\[
	\int\limits_{-r}\limits^r \frac{|T_k(t/r)|}{\sqrt{r^2 - t^2}} dt   
	\leq  	\int\limits_{-r}\limits^r \frac{dt}{\sqrt{r^2 - t^2}} = \pi.
\]
Recalling  the definition of $\mathcal{C}_{R,n}$	and using  the above two formulas and \eqref{bound_T},
%\eqref{def_C}, \eqref{linear_C}, \eqref{est_Tn}
% and taking into account $\lambda>1$, 
 we get that
\begin{equation*}%\label{eq_Fw2}
	\begin{aligned}
	 &\left\|	\mathcal{C}_{R,n}[w] - \mathcal{C}_{R,n}[\mathcal{F} v] \right\|_{\calL^{\infty}(B_R)} 
	 \\
	   & \qquad \qquad  \qquad  \leq \frac{2}{\pi}\sum\limits_{k=0}\limits^{n-1} \|T_k\|_{\calL^{\infty}([-R/r,R/r])} \|w - \mathcal{F}v\|_{\calL^{\infty}(B_r)} 
	     \int\limits_{-r}\limits^r \frac{|T_k(t/r)|}{\sqrt{r^2 - t^2}} dt \\
	    &\qquad  \qquad  \qquad \leq  2 \sum\limits_{k=0}\limits^{n-1}\left(\dfrac{2R}{r}\right)^k  \delta \leq  2 \left(\dfrac{2R}{r}\right)^n \delta.
	 \end{aligned}
\end{equation*}
This bound together with \eqref{eq_Fw1} implies  Lemma \ref{Lemma_C}.
%%%%%%%%%%%%%%%%%%%%%%%%%%%%%%%%%%%%%%%%%%%%%%%%%%%%%%%%%%%%%%%%%%%%%%%%%%%%%%%%%%%%%%%%%%%%%%%%%%%%%%%%%%%%
%%%%%%%%%%%%%%%%%%%%%%%%%%%%%%%%%%%%%%%%%%%%%%%%%%%%%%%%%%%%%%%%%%%%%%%%%%%%%%%%%%%%%%%%%%%%%%%%%%%%%%%%%

%\section*{Acknowledgements}
%The second author was partially supported by the Russian Federation Goverment grant No. 2010-220-01-077.


\begin{thebibliography}{99}

\bibitem{Alessandrini1988}
 G. Alessandrini, 
  Stable determination of conductivity by boundary measurements, 
 \textit{Applicable Analysis}, \textbf{27} (1988), 153--172.

%\bibitem{AV2005}
%G. Alessandrini, S. Vassella,
%{\it Lipschitz stability for the inverse conductivity problem},
%Adv. in Appl. Math. 35, 2005, no.2, 207-241.

%\bibitem{ABR2008}
%N.V. Alexeenko, V.A. Burov, O.D. Rumyantseva, 
%{\it Solution of three-dimensional
%acoustical inverse scattering problem,II: modified Novikov algorithm}, Acoust. J. 54(3),
%2008, 469--482 (in Russian); English transl.: Acoust. Phys. 54(3), 2008, 407--419.
%

\bibitem{BLT2010}
G.~Bao, J.~Lin, F.~Triki,
A multi-frequency inverse source problem,
\textit{Journal of Differential Equations}, \textbf{249} (2010), 3443--3465.


%\bibitem{Beals1985}
%R. Beals, R. Coifman, {\it Multidimensional inverse scattering and nonlinear partial differential equations},
%Proc. Symp. Pure Math., 43, 1985, 45-70.

%\bibitem{Begehr2010}
%H. Begehr and T. Vaitekhovich, {\it Some harmonic Robin functions in the complex plane},
%Adv. Pure Appl. Math. 1, 2010, 19–34.

%\bibitem{Berezin1991}
%F. A. Berezin and M. A. Shubin, {\it The Schr\"odinger Equation}, Vol. 66 of Mathematics and Its Applications, Kluwer
%Academic, Dordrecht, 1991.

%\bibitem{Berezanskii1958}
%Yu.M. Berezanskii, { \it The uniqueness theorem in the inverse problem of spectral analysis for the Schrodinger equation.\it} 
%(Russian) Trudy Moskov. Mat. Obsc. 7 (1958) 1–62.

%\bibitem{BK2012}
%L. Beilina, M.V. Klibanov, 
%{\it Approximate global convergence and adaptivity for coefficient inverse problems}, Springer (New York), 2012. 407 pp.

%\bibitem{Buckhgeim2008}
% A. L. Buckhgeim, 
% {\it Recovering a potential from Cauchy data in the two-dimensional
%case}, J. Inverse Ill-Posed Probl. 16, 2008, no. 1, 19-33.


%\bibitem{Calderon1980} A.P. Calder\'on,  {\it On an inverse boundary problem}, Seminar on Numerical Analysis
%and its Applications to Continuum Physics, Soc. Brasiliera de Matematica, Rio de Janeiro, 1980, 61--73.

%\bibitem{Chadan1977}
%K. Chadan, P.C. Sabatier, {\it Inverse problems in quantum scattering theory.} With a foreword by R. G. Newton. Texts and Monographs %in Physics. Springer-Verlag, New York-Berlin, (1977). xxii+344 pp. 
 
%
%\bibitem{CR2003}
%M. Di Cristo, L. Rondi
%{\it Examples of exponential instability for inverse
%inclusion and scattering problems}
%{ Inverse Problems} 19, 2003, 685--701.


\bibitem{DT2019}
L. Demanet,  A. Townsend,
Stable Extrapolation of Analytic Functions,
\textit{Foundations of Computational Mathematics}, \textbf{19} (2019), 297--331.

%\bibitem{Druskin1982}
%V. Druskin, {\it The unique solution of the inverse problem in electrical surveying and
%electrical well logging for piecewise-constant conductivity}, Physics of the Solid Earth 18(1), 1982, 51-53.

%\bibitem{Faddeev1956}
%L.D. Faddeev,{\it Uniqueness of the solution of the inverse scattering problem.} Vestn. Leningr. Univ. 7 126–30 
%(in Russian), 1956.

%\bibitem{Faddeev1965}
%L.D. Faddeev, {\it Growing solutions of the Schr\"odinger equation,} Dokl. Akad. Nauk SSSR, 165, N.3, 1965, 514-517 (in Russian);
%English Transl.: Sov. Phys. Dokl. 10, 1966, 1033-1035.
%
%\bibitem{Faddeev1974}
%L.D. Faddeev, 
%{\it The inverse problem in the quantum theory of scattering. II,} Current problems in mathematics,  Vol. 3, 1974, pp. 93-180, 259.  Akad. Nauk SSSR Vsesojuz. Inst. Naucn. i Tehn. Informacii, Moscow(in Russian); English Transl.: J.Sov. Math. 5, 1976, 334-396.
%
%\bibitem{Gelfand1954}
%I.M. Gelfand, {\it Some problems of functional analysis and algebra}, Proceedings of the
%International Congress of Mathematicians, Amsterdam, 1954, pp.253-276.

%\bibitem{GM2009}
%F. Gesztesy and M. Mitrea, {\it Robin-to-Robin Maps and Krein-Type Resolvent Formulas for Schrodinger Operators on Bounded Lipschitz 
%Domains},
%Modern Analysis and Applications 
%Operator Theory: Advances and Applications, Volume 191, 2009, Part 1, 81-113.

%\bibitem{Grinevich2000}
%P.G. Grinevich, {\it The scattering transform for the two-dimensional
%Schr\"odinger operator with a potential that decreases at infinity at fixed
%nonzero energy}, Uspekhi Mat. Nauk 55:6(336),2000, 3–70 (Russian);
%English translation: Russian Math. Surveys 55:6, 2000, 1015–1083.


%\bibitem{Grinevich1988}
%P.G. Grinevich, S.P. Novikov, {\it Two-dimensional "inverse scattering
%problem" for negative energies and generalized-analytic functions. 1.
%Energies below the ground state}, Funct. Anal. i ego Pril.22:1, 1988,
%23–33 (Rusian); English translation: Funct. Anal. Appl. 22, 1988, 19–27.

%\bibitem{GN2012}
%P.G. Grinevich, R.G. Novikov, {\it 
%Faddeev eigenfunctions for point potentials in two dimensions}, 
%Physics Letters A 376, 2012, 1102-1106

\bibitem{HH2001}
P. H\"ahner, T. Hohage, New stability estimates for the inverse acoustic inhomogeneous medium problem and applications, 
\textit{SIAM Journal on Mathematical Analysis},  \textbf{33}(3) (2001), 670--685.



\bibitem{HW2017}
T.~Hohage, F.~Weidling, Variational source conditions and stability estimates for inverse electromagnetic medium scattering problems,
\textit{Inverse Problems and Imaging},  \textbf{11}(1) (2017), 203--220.

%\bibitem{Henkin1987}
%G.M. Henkin and R.G. Novikov, 
%{\it The $\bar{\partial}$-equation in the multidimensional inverse scattering problem}, 
%Uspekhi Mat. Nauk 42(3), 1987, 93--152 (in Russian); English
%Transl.: Russ. Math. Surv. 42(3), 1987, 109--180.

%\bibitem{Isaev2011}
%M.I. Isaev,
%{\it Exponential instability in the Gel'fand inverse 
%problem on the energy intervals}, J. Inverse Ill-Posed Probl., Vol. 19(3), 2011, 453--473.


%\bibitem{Isaev2013} M.I. Isaev, {\it Instability in the Gel'fand inverse problem at high energies},
%Applicable Analysis, Vol. 92, No. 11, 2013, 2262--2274.

\bibitem{Isaev2013+}
M. Isaev,  Energy and regularity dependent stability
estimates for near-field inverse scattering in  multidimensions, \textit{Journal of Mathematics}, (2013), Article ID 318154, 10~p.; DOI:10.1155/2013/318154.

%\bibitem{Isaev2013++}
%M. Isaev, Exponential instability in the inverse scattering problem on the energy interval,
%\textit{Functional analysis and its applications},  \textbf{47}(3)  (2013), 187--194.


%\bibitem{IN2012}
%M.I. Isaev, R.G. Novikov
%{\it Energy and regularity dependent stability estimates for the Gel'fand
%inverse problem in multidimensions}, J. of Inverse and Ill-posed Probl., Vol. 20(3), 2012,  313--325.
%

%\bibitem{IN2013}
%M.I. Isaev, R.G. Novikov, { \it Stability estimates for determination of potential from the impedance
%boundary map}, Algebra and Analysis, Vol. 25(1), 2013,  37--63 (in Russian); Engl. Transl.:  
%St. Petersburg Mathematical Journal, Vol. 25, 2014, 23--41.

%\bibitem{IN2013+}
%
% M.I. Isaev, R.G. Novikov, {\it Reconstruction of a potential from the impedance
%boundary map}, Eurasian Journal of Mathematical and Computer Applications, Vol. 1(1), 2013,  5--28.

\bibitem{IN2013++}
M. Isaev, R.G. Novikov,  New global stability estimates for monochromatic
inverse acoustic scattering, \textit{SIAM Journal on Mathematical Analysis}, \textbf{45}(3) (2013), 1495--1504.



\bibitem{IN2020}
M. Isaev, R.G. Novikov,  H\"older-logarithmic stability in the Fourier analysis, e-preprint  arXiv:2005.01414.


\bibitem{Isakov2011}
V. Isakov, Increasing stability for the  Schr\"odinger potential from the Dirichlet-to-Neumann map,
\textit{Discrete and Continuous Dynamical Systems},  \textbf{4}(3)  (2011), 631--640.

%\bibitem{INUW2014}
%V. Isakov, S. Nagayasu, G. Uhlmann, J.-N. Wang, 
%{\it Increasing stability of the inverse boundary value problem
%for the Schr\"odinger equation}, Contemporary Mathematics (to appear), e-print arXiv:1302.0940.
%

\bibitem{LRS1986}
M.M. Lavrent'ev,  V.G. Romanov, S.P. Shishatskii, 
Ill-posed problems of mathematical physics and analysis, Translated from the Russian by J. R. Schulenberger.   \textit{Translations of Mathematical Monographs}, \textbf{64} (1986).
American Mathematical Society,  Providence,  R.I.,  vi+290 pp.


%\bibitem{Mandache2001}
% N. Mandache,
%{Exponential instability in an inverse problem for the Schr\"odinger equation}, 
% \textit{Inverse Problems}, \textbf{17}  (2001), 1435--1444.  



%\bibitem{KV1985}
%R. Kohn, M. Vogelius, {\it Determining conductivity by boundary measurements II}, Interior
%results, Comm. Pure Appl. Math. 38, 1985, 643-667.

%\bibitem{Lanzani2004}
%L. Lanzani and Z. Shen, { \it On the Robin boundary condition for
%Laplace's equation in Lipschitz domains}, Comm. Partial Differential Equations, 29, 2004, 91–109.
%
%\bibitem{LN1987}
%R.B. Lavine and A.I. Nachman, {\it On the inverse scattering transform of the $n$-dimensional Schr\"odinger}
%operator Topics in Soliton Theory and Exactly Solvable Nonlinear Equations ed M Ablovitz, B Fuchssteiner
%and M Kruskal (Singapore: World Scientific), 1987, pp 33-44.

%
%\bibitem{LR1986}
%M.M. Lavrent'ev,  V.G. Romanov, S.P. Shishatskii, 
%{\it Ill-posed problems of mathematical physics and analysis}, Translated from the Russian by J. R. Schulenberger.
%Translation edited by Lev J. Leifman. Translations of Mathematical Monographs, 64.
%American Mathematical Society, Providence, RI, 1986. vi+290 pp.

%\bibitem{Liu1997}
%L. Liu, 
%{\it Stability Estimates for the Two-Dimensional Inverse Conductivity Problem},
%Ph.D. thesis, Department of Mathematics, University of Rochester, New York, 1997.

%\bibitem{Kolmogorov1959}
%A.N. Kolmogorov, V.M. Tikhomirov, 
%{ \it $\epsilon$-entropy and $\epsilon$-capacity in functional spaces} Usp. Mat. Nauk 14(1959)
%3–86 (in Russian) (Engl. Transl.  Am. Math. Soc. Transl. 17 (1961) 277–364)



%\bibitem{Newton1989}
%R.G. Newton,   
%{ \it Inverse Schr\"odinger scattering in three dimensions.} Texts and Monographs in Physics. Springer-Verlag, Berlin, 1989. x+170 pp.
%\bibitem{Watson1944}
%G. N. Watson,  
%{ \it A Treatise on the Theory of Bessel Functions.} 
%Cambridge University Press, Cambridge, England; The Macmillan Company, New York, 1944.
%\bibitem{Nachman1988}
%A. Nachman, {\it Reconstructions from boundary measurements}, Ann. Math. 128, 1988, 531-576

%\bibitem{Nachman1996}
%A. Nachman, {\it Global uniqueness for a two-dimensional inverse boundary value problem},
%Ann. Math. 143, 1996, 71-96.

%\bibitem{NUW2011}
%S. Nagayasu, G. Uhlmann, J.-N. Wang, 
%{\it Increasing stability in an inverse problem for the acoustic equation},
%Inverse Problems  29, 2013, 025013(11pp).  
%
%
%
%\bibitem{Novikov1988}
%R.G. Novikov, {\it Multidimensional inverse spectral problem for the equation $-\Delta \psi + (v(x) - Eu(x))\psi = 0$,}
%Funkt. Anal. Prilozhen. 22(4), 1988, 11-22 (in Russian); Engl. Transl.  Funct. Anal. Appl. 22, 1988, 263-272.

%\bibitem{Novikov1992}
%R.G. Novikov, { \it The inverse scattering problem on a fixed energy level for the two-dimensional
%Schr\"odinger operator}, J.Funct. Anal. 103 (1992), 409-463.

%\bibitem{Novikov1996}
%R.G. Novikov,  {\it $\bar{\partial}$-method with nonzero background potential. Application to inverse
%scattering for the two-dimensional acoustic equation}, Comm. Partial Differential
%Equations 21, 1996, no. 3-4, 597-618.
%
%\bibitem{Novikov1998}
%R.G. Novikov, {\it Rapidly converging approximation in inverse quantum scattering in dimension
%2}, Physics Letters A 238, 1998, 73-78.

%\bibitem{Novikov1999}
%R.G. Novikov, { \it Approximate solution of the inverse problem of quantum scattering theory with fixed energy in dimension 2}, 
%Proceedings of the Steklov Mathematical Institute 225, 1999, 
%Solitony Geom. Topol. na Perekrest., 301-318 (in Russian); Engl. Transl. in Proc. Steklov Inst. Math.
%225, 1999, no. 2, 285-302.

 

%\bibitem{Novikov2005+}
%R.G. Novikov,
%{ \it The $\bar\partial$-approach to approximate inverse scattering at fixed energy in three dimensions. }
% IMRP Int. Math. Res. Pap. 2005, no. 6, 287-349. 
%
%\bibitem{Novikov2005}
%R.G. Novikov, {\it Formulae and equations for finding scattering data from the Dirichlet-to-Neumann map with nonzero background potential}, 
%Inverse Problems 21, 2005, 257-270.

%\bibitem{Novikov2006}
%R.G. Novikov, { \it On non-overdetermined inverse scattering at zero energy in three dimensions },
%Ann. Scuola Norm. Sup. Pisa Cl. Sci. 5, 2006, 279-328

%\bibitem{Novikov2008}
%R.G. Novikov, 
%{ \it The $\bar\partial$-approach to monochromatic inverse scattering  in three dimensions,}
%J. Geom. Anal 18, 2008, 612-631.
%
%\bibitem{Novikov2009}
%R.G. Novikov, {\it An effectivization of the global reconstruction in the Gel'fand-Calderon
%inverse problem in three dimensions}, Contemporary Mathematics, 494, 2009, 161-184.

%\bibitem{Novikov1994}
%R.G. Novikov,
%{ \it The inverse scattering problem at fixed energy for the three-dimensional Schrodinger equation with an exponentially decreasing potential.}
% Comm. Math. Phys. 161 (1994), no. 3, 569-595. 

%\bibitem{Novikov2001}
%R.G. Novikov,
%{ \it On determination of the Fourier transform of a potential from the scattering amplitude. }
% Inverse Problems 17 (2001), no. 5, 1243–1251. 


\bibitem{Novikov2011}
R.G. Novikov,
New global stability estimates for the Gel'fand-Calderon inverse problem, 
\textit{Inverse 
Problems},  \textbf{27} (2011), 015001,  21pp.
%
%\bibitem{Novikov2013}
%R.G. Novikov,
%{\it Approximate Lipschitz stability for non-overdetermined inverse scattering at fixed energy},
%J.Inverse  Ill-Posed Probl., Vol. 21(6), 2013,  813--823. 

\bibitem{Novikov2020} R.G. Novikov, Multidimensional inverse scattering for the Schr\"odinger equation, e-preprint:
https://hal.archives-ouvertes.fr/hal-02465839v1.

%
%%\bibitem{NN2009}
%R.G. Novikov and N.N. Novikova, { \it On stable determination of potential by boundary
%measurements}, ESAIM: Proceedings 26, 2009, 94-99.
%
%\bibitem{NS2010} R. Novikov and M. Santacesaria, {\it A global stability estimate for 
%the Gel'fand- Calderon inverse problem in two dimensions}, J.Inverse 
%Ill-Posed Probl., Volume 18, Issue 7, 2010, 765--785. 

%\bibitem{NS2011}
%R. Novikov and M. Santacesaria, 
%{\it Global uniqueness and reconstruction for the multi-channel Gel'fand-Calderon inverse problem in two dimensions},
%Bulletin des Sciences Mathematiques 135, 5, 2011, 421-434.

%\bibitem{NS2012}
%R. Novikov and M. Santacesaria, 
%{\it Monochromatic Reconstruction Algorithms for Two-dimensional Multi-channel Inverse Problems},
%International Mathematics Research Notes, 2012, doi: 10.1093/imrn/rns025.

%
%\bibitem{PZ2013}
%L. P\"aiv\"arinta, M. Zubeldia,
%{\it The inverse Robin boundary value problem in a half-space},
%e-print arXiv:1311.6947

%\bibitem{Rondi2006}
%L. Rondi, {\it A remark on a paper by Alessandrini and Vessella}, Adv. in Appl. Math. 36 (1), 2006, 67-69.

%\bibitem{RS2014}
%L. Rondi, M. Sini, 
%{\it Stable determination of a scattered wave from its far-field pattern: the high frequency asymptotics,}
%e-print arXiv:1401.4081.

%\bibitem{S2011}
%M. Santacesaria, 
%{\it Global stability for the multi-channel Gel'fand–Calderon inverse problem in two dimensions},
%Bulletin des Sciences Mathematiques,  doi:10.1016/j.bulsci.2012.02.004,  
% e-print: hal-00569366.
 
%\bibitem{S2013}
%M. Santacesaria, 
%{\it New global stability estimates for the Calderon inverse problem in two dimensions},
% Journal of the Institute of Mathematics of Jussieu, Volume 12, Issue 3, 2013, 553-569.
 
%\bibitem{Santacesaria2012} 
%M. Santacesaria, {\it
%Stability estimates for an inverse problem
%for  the Schr\"odinger equation at negative
%energy in two dimensions},
%Applicable Analysis, 2013, Vol. 92, No. 8, 1666--1681.
 
 
\bibitem{Santacesaria2015}
M. Santacesaria, 
{A Holder-logarithmic stability estimate for an inverse problem in two dimensions},
\textit{Journal of Inverse and Ill-posed Problems}, \textbf{23}(1) (2015), 51--73.

%\bibitem{Stefanov1990}
%P. Stefanov, {\it A uniqueness result for the inverse back-scattering problem,} Inverse Problems, 6,
%1990, 1055-1064 
%  
%\bibitem{SU1987}
%J. Sylvester and G. Uhlmann,  {\it A global uniqueness theorem for an inverse boundary
%value problem}, Ann. of Math. 125, 1987, 153-169.

%\bibitem{Stefanov1990}
%P. Stefanov, {\it A uniqueness result for the inverse back-scattering problem,} Inverse Problems, 6,
%1990, 1055-1064
%\bibitem{Stefanov1990}
%	P. Stefanov,
%	{\it Stability of the inverse problem in potential scattering at fixed energy}
%Annales de l'institut Fourier, tome 40, N4 (1990), p.867-884. 
%\bibitem{Vekua1962}
%I.N. Vekua, {\it Generalized Analytic Functions}, Pergamon Press Ltd., 1962.
%


\bibitem{Tuan2000}
V.K.~Tuan, Stable analytic continuation using hypergeometric summation,
\textit{Inverse Problems}, \textbf{16}(1) (2000), 75--87.

%\bibitem{Weder1991}
%R. Weder, 
%{\it Generalized limiting absorption method and multidimensional inverse scattering theory},  
%Mathematical Methods in the Applied Sciences, 14, 1991, 509-524.

%\bibitem{Uhlmann2009}
%G. Uhlmann, {\it Electrical impedance tomography and Calderon's
%problem}, Inverse Problems 25(2009), 123011.


\bibitem{Vessella1999}
S.~Vessella, A continuous dependence result in the analytic continuation
problem, Forum Mathematicum, \textbf{11}(6) (1999), 695--703.


\end{thebibliography}
\end{document}